\documentclass{amsart}[12pt]
\usepackage{amsmath,amsfonts,amssymb,amscd,verbatim,delarray,verbatim,epigraph}
\DeclareMathSymbol{\twoheadrightarrow} {\mathrel}{AMSa}{"10}

\def\bbQ{{\mathbb Q}}

         \def\bbb{\mathbf{b}}

        \def\cC{{\mathfrak C}}
        \def\cV{{\mathcal C}}

\def\bbZ{{\mathbb Z}}
\def\bbC{{\mathbb C}}
\def\bbR{{\mathbb R}}

\def\Gal{\mathrm{Gal}}

\def\ord{\mathrm{ord}}

\def\End{\mathrm{End}}

\def\Aut{\mathrm{Aut}}
\def\Hom{\mathrm{Hom}}
\def\Mat{\mathrm{Mat}}

\def\cJ{{\mathcal J}}

        \def\K_a{\bar{K}}
\def\dim{\mathrm{dim}}
\def\OC{{\mathcal O}}

\def\fm{{\mathfrak m}}
\def\cZZ{{\mathcal Z}}

\newtheorem{thm}{Theorem}[section]
\newtheorem{lem}[thm]{Lemma}
\newtheorem{cor}[thm]{Corollary}
\newtheorem{prop}[thm]{Proposition}
\theoremstyle{definition}
\newtheorem{defn}[thm]{Definition}
\newtheorem{ex}[thm]{Example}

\newtheorem{rem}[thm]{Remark}

\newtheorem{sect}[thm]{}

\begin{document}
\title[Jacobians with with automorphisms of prime order]{Jacobians with automorphisms of prime order}
\author{Yuri G. Zarhin}
\address{Department of Mathematics, Pennsylvania
State University, University Park, PA 16802, USA}
\email{zarhin\char`\@math.psu.edu}

\thanks{The  author  was partially supported by Simons Foundation Collaboration grant   \# 585711  and the Travel Support for Mathematicians Grant MPS-TSM-00007756 from the Simons Foundation. This paper  was finished in August-September 2025 during his stay at Max-Planck-Institut f\"ur Mathematik (Bonn, Germany), whose hospitality and support are gratefully acknowledged.}
\begin{abstract}
In this paper we study principally polarized complex abelian varieties $(X,\lambda)$ that admit an automorphism $\delta$ of
prime order $p>2$. It turns out that certain natural conditions on the multiplicities of the action of $\delta$ on $\Omega^1(X)$
do guarantee that those polarized varieties are not canonically polarized
jacobians of curves. 
\end{abstract}

\maketitle

\section{Introduction}
The work of Clemens and Griffiths \cite{GC}  on intermediate jacobians of threefolds  and their applications to the L\"{u}roth problem increased an interest in the classical
 question - how to find out  that a  given  principally polarized $g$-dimensional complex abelian variety $(X,\lambda)$
is not isomorphic to the  canonically polarized jacobian $(\cJ,\Theta)$  of a smooth irreducible projective curve $\cV$ of genus $g$.  In this paper we address this question in the case when $(X,\lambda)$ admits an additional symmetry - an automorphism
$\delta$ of prime period $p>2$ that satisfies the $p$th cyclotomic equation. We  give our answer in terms of the {\sl multiplicity function}
$\mathbf{a}_{X,\delta}$, which assigns to each $h \in (\bbZ/p\bbZ)^{*}$ the multiplicity  $\mathbf{a}_X(h)$ of the eigenvalue $\zeta_p^h$ of
the linear operator $\delta_{\Omega}$ in  the space $\Omega^1(X)$ of differentials of the first kind on $X$  induced by $\delta$. Namely, we describe explicitly  all integer-valued 
 (we call them {\sl strongly admissible})
functions  $f: (\bbZ/p\bbZ)^{*} \to \bbZ_{+}$ that enjoy the following property:  there exists  a triple $(\cJ,\Theta, \delta)$ as above such that $f=\mathbf{a}_{\cJ,\delta}$.

It turns out that not every function $\mathbf{a}_{X,\delta}$ coincides with some $\mathbf{a}_{\cJ,\delta}$. For example, if $p=3$ then there are precisely
$(g+1)$ functions of type $\mathbf{a}_{X,\delta}$; however, there are approximately only $g/3$ functions  of type $\mathbf{a}_{\cJ,\delta}$ (Section \ref{Peq3}, see also \cite{ZarhinManinFest}).

The   paper is organized as follows. In Section \ref{s1} we study canonically polarized jacobians $(\cJ,\Theta)$ of curves $\cV$ such  that there is automorphism $\delta\in \Aut(\cJ,\Theta)$ with
properties described above. It turns out that $\delta$ is induced by an automorphism $\delta_{\cV}$ of  $\cV$ of order $p$.  It turns out that the number of fixed points of $\delta_{\cV}$
is $\frac{2g}{p-1}+2$ and each of these points $P$  is nondegenerate, i.e., its {\sl index} $\epsilon_P$ is a primitive $p$th root of unity.
This gives rise to the integer-valued function
$\mathbf{b}: (\bbZ/p\bbZ)^{*} \to \bbZ_{+}$ that assigns to each $h \in (\bbZ/p\bbZ)^{*} $ the number of fixed points of $\delta_{\cV}$   with index $\zeta_p^h$.
Our main result (Theorem \ref{nonjacob}) expresses explicitly the function  $\mathbf{a}_{\cJ,\delta}$ in terms of the function $\mathbf{b}$, which imposes
the restrictions on the function  $\mathbf{a}_{\cJ,\delta}$.  In Section \ref{construction} we prove Theorem \ref{inverse}, which, in particular,  asserts that the
necessary condition for a function $f:  (\bbZ/p\bbZ)^{*} \to \bbZ_{+}$ to coincide with some $\mathbf{a}_{\cJ,\delta}$ imposed by Theorem \ref{inverse} is actually sufficient.
In Section \ref{Peq3} we discuss in detail the case $p=3$.  Section \ref{CMcase} deals with CM abelian varieties of dimension $(p-1)/2$. In Section \ref{RosatiS}
we discuss  certain principally polarized abelian varieties that are not isomorphic as an allgebraic variety to jacobians.

{\bf Acknowledgements} I am grateful to George Andrews for a very informative letter about partitions.
I thank the referees for thoughtful comments that helped to improve the exposition.

\section{Principally polarized abelian varieties with automorphisms}
\label{s1}

We write $\bbZ_{+}$ for the set of {\sl nonnegative} integers, $\bbQ$ for the field of rational numbers
and $\bbC$ for the field of complex numbers. We have
$$\bbZ_{+}\subset \bbZ \subset \bbQ\subset \bbR\subset \bbC$$
where $\bbZ$ is the ring of integers and $\bbR$ is the field of real numbers. If $B$ is a finite (may be, empty) set then we write $\#(B)$ for its cardinality.
Let $p$ be an odd prime and $\zeta_p \in \bbC$ a primitive (complex) $p$th root of unity.
It generates  the multiplicative order $p$ cyclic group $\mu_p$ of $p$th roots of unity.
We write $\bbZ[\zeta_p]$ and $\bbQ(\zeta_p)$ for the $p$th cyclotomic ring and the $p$th cyclotomic field respectively. We have
$$\zeta_p\in \mu_p\subset \bbZ[\zeta_p]\subset \bbQ(\zeta_p)\subset \bbC.$$

Let $g\ge 1$ be an  integer and  $(X,\lambda)$ a principally polarized $g$-dimensional
abelian variety over the field $\bbC$ of complex numbers, $\delta$ an automorphism of
$(X,\lambda)$ that satisfies the cyclotomic equation
\begin{equation}
\label{cyclotomic}
\sum_{j=0}^{p-1}\delta^j=0 \in \End(X).
\end{equation}
 In other words, $\delta$ is a periodic automorphism
of order $p$, whose set of fixed points is finite. This gives rise to the embeddings
$$\bbZ[\zeta_p]\hookrightarrow \End(X), 1\mapsto 1_X,\ \zeta_p\mapsto \delta,$$
$$\bbQ(\zeta_p)\hookrightarrow \End^{0}(X), 1\mapsto 1_X,\ \zeta_p\mapsto \delta.$$
(Hereafter we write $1_X$ for the identity automorphism of $X$.)
Since the degree $[\bbQ(\zeta_p):\bbQ]=p-1$, it follows from  \cite[Ch. 2, Prop. 2]{ST} (see also \cite[Part II, p. 767]{Ribet}) that
\begin{equation}
\label{divide}
(p-1) \mid 2g.
\end{equation}
By functoriality, $\bbQ(\zeta_p)$ acts on the $g$-dimensional complex
vector space $\Omega^1(X)$ of  differentials of the first kind on
$X$. This endows $\Omega^1(X)$ with the structure of a

\noindent $\bbQ(\zeta_p)\otimes_{\bbQ}\bbC$-module. Clearly,
 $$\bbQ(\zeta_p)\otimes_{\bbQ}\bbC=\oplus_{j=1}^{p-1} \bbC$$
 where the $j$th summand corresponds to the field embedding $\bbQ(\zeta_p) \hookrightarrow \bbC$ that sends
 $\zeta_p$ to $\zeta_p^{j}$. So, $\bbQ(\zeta_p)$ acts on
 $\Omega^1(X)$ with multiplicities $a_j$ ($j=1, \dots p-1$). Clearly, all
 $a_j$  are nonnegative integers and
 \begin{equation}
 \label{sumA}
 \sum_{j=1}^{p-1}a_j=g.
 \end{equation}
 
 In addition,  
 \begin{equation}
 \label{Hsymmetry}
 a_j+a_{p-j} =\frac{2g}{p-1} \ \forall j=1, \dots p-1;
 \end{equation}
  this is a special case of a general well known result about endomorphism fields of complex abelian varieties: see, e.g., \cite[p. 84]{MoonenZarhin}.
 (See also Remark \ref{j0} below where another proof for jacobians is given.)
 We may view the collection $\{a_j\}$ as a nonnegative integer-valued function
  $$\mathbf{a}=\mathbf{a}_{X,\delta}$$
   on the finite cyclic group $G=(\bbZ/p\bbZ)^{*}$ of order $(p-1)$ where
 \begin{equation}
 \label{functionA}
\mathbf{a}_{X,\delta}(j\bmod p)= \mathbf{a}(j\bmod p)=a_j  \ (1 \le j \le p-1);  \quad  \sum_{h\in G}\mathbf{a}_{X,\delta}(h)=g.
\end{equation}
The group $G$ contains two distinguished elements, namely, the identity element $1\bmod p$ and the only element
$(-1)\bmod p = (p-1)\bmod p$ of order $2$. If $h$ is an element of $G$ then we write $-h$ for the product 
of $h$ by
$(-1)\bmod p$ in $G$.
If $h =j\bmod p$ then $-h=(p-j)\bmod p$. In light of \eqref{Hsymmetry},
\begin{equation}
\label{Gsymmetry}
\mathbf{a}_{X,\delta}(h)+\mathbf{a}_{X,\delta}(-h)=\frac{2g}{p-1} \ \forall  h \in G.
\end{equation}
For each $h=j \bmod p\in G$ we write 
$$\zeta_p^h:=\zeta_p^j,$$
which is a primitive $p$th root of unity that does {\sl not} depend on the choice of $j$.

\begin{defn}
Let $p$ be an odd prime and $g$ a positive integer such that $p-1$ divides $2g$.
Let $f: G \to \bbZ_{+}$ be a nonnegative integer-valued function.
\begin{itemize}
\item[(i)]
We say that $f$ is {\sl well rounded} of degree $g$ if 
$$f(h)+f(-h)=\frac{2g}{p-1} \ \forall  h \in G.$$
\item[(ii)]
We say that $f$ is {\sl  admissible } of degree $g$  if there exist a principally polarized $g$-dimensional complex abelian variety $(X,\lambda)$ with an 
automorphism $\delta\in \Aut(X,\lambda)$
of period $p$ 
such that  \eqref{cyclotomic} holds and
$$f=\mathbf{a}_{X,\delta}.$$
\item[(iii)]
We say that $f$ is {\sl strongly admissible} of degree $g$ if it is  admissible of degree $g$ and one may choose the corresponding $(X,\lambda,\delta)$ 
in such a way that $(X,\lambda)$
is the canonically polarized jacobian of a smooth irreducible connected projective curve of genus $g$.
\end{itemize}
\end{defn}

\begin{rem}
\label{round}
\begin{enumerate}
\item[(i)]
In light of \eqref{Gsymmetry}, our $\mathbf{a}=\mathbf{a}_{X,\delta}$ is well rounded. In other words, every  admissible function is well rounded.
\item[(ii)]
The number of well rounded functions (for given $g$ and $p$) is  obviously
$$\left(\frac{2g}{p-1}+1\right)^{(p-1)/2}.$$
\item[(iii)]
Let $f: G \to \bbZ_{+}$ be a well rounded function of degree $g$.  Then
$$\bar{f}:  G \to \bbZ_{+}, \ h \mapsto f(-h)=\frac{2g}{p-1}-f(h)$$
is also well rounded of degree $g$.
In addition, if $f$ is admissible (resp. strongly admissible) then
$\bar{f}$ is also  admissible (resp. strongly admissible). Namely, let
$(X,\lambda)$ be a principally polarized abelian variety, $\delta \in \Aut(X,\lambda)$ an automorphism of period $p$ that satisfies the $p$th cyclotomic equation \eqref{cyclotomic} in $\End(X)$ such that $f$ coincides with the corresponding multiplicity function $\mathbf{a}_{X,\delta}$. Then $\delta^{-1}$ is also an automorphism of $\Aut(X,\lambda)$ of period $p$ that satisfies the $p$th cyclotomic equation and such that
$\bar{f}=\mathbf{a}_{X,\delta^{-1}}$.
So, $\bar{f}$ is admissible. In addition, if $(X,\lambda)=(\mathcal{J}(\mathcal{C}),\Theta)$ is the canonically polarized jacobian of a curve  $\mathcal{C}$ then 
$\bar{f}=\mathbf{a}_{\mathcal{J}(\mathcal{C}),\delta^{-1}}$
is strongly admissible.
\end{enumerate}
\end{rem}

\begin{ex}
\label{p3weak}
Let $p=3$ and $E$ an elliptic curve over $\bbC$ with complex multiplication by $\bbZ[\zeta_3]$. We may take as $E$  the smooth projective model of $y^2=x^3-1$
where $\zeta_3$ acts on $E$ by an automorphism
$$\delta_E: (x,y) \mapsto (\zeta_3 x, y).$$
Clearly, $\delta_E$ satisfies the 3rd cyclotomic equation and respects the only principal polarization on $E$. 
We have
$$\Omega^1(E)=\bbC \cdot \frac{dx}{y}, \quad (\delta_E)_{\Omega}\frac{dx}{y}=\frac{ d(\zeta_3 x)}{y}=\zeta_3 \frac{dx}{y}.$$
This means that
$${\mathbf a}_{E,\delta_E}(1 \bmod 3)=1, \quad {\mathbf a}_{E,\delta_E}(2 \bmod 3)=0.$$
Let $g$ be a positive integer, and $f(1)$ and $f(2)$ are nonnegative integers,
whose sum is $g$. Let us put 
$$Y_1=E^{f(1)}, \quad Y_2=E^{f(2)},  \quad  Y=Y_1\times Y_2.$$
Let $\lambda_Y$ be the principal polarization on $Y$ that is the product of $g$ pull-backs of the principal polarization on $E$.
Let us consider the automorphism $\delta_Y$ of $Y$ that acts (diagonally) as $\delta_E$ on $Y_1=E^{f(1)}$ and as $\delta_E^2=\delta_E^{-1}$ on $Y_2=E^{f(2)}$.
Clearly, $\delta_Y$ satisfies the 3rd cyclotomic equation and respects $\lambda_Y$. It is also clear that
$$\mathbf{a}_{Y,\delta_Y}(1\bmod 3)=f(1), \quad \mathbf{a}_{Y,\delta_Y}(2\bmod 3)=f(2).$$
In other words, if $p=3$ then every well rounded function of degree $g$ is  admissible.
We will see (Section \ref{Peq3} below) that not every such function is strongly admissible.
\end{ex}

We will also need the function
\begin{equation}
\label{jG}
\mathbf{j}: G=(\bbZ/p\bbZ)^{*} \to \bbZ, \  (j\bmod p) \mapsto j \ (1 \le j \le p-1).
\end{equation}
Clearly,
\begin{equation}
\label{jmult}
\mathbf{j}(h_1 h_2)\equiv \mathbf{j}(h_1)\mathbf{j}(h_2) \bmod p \ \forall h_1,h_2\in G.
\end{equation}

Recall that if $f_1(h)$ and $f_2(h)$ are complex-valued functions on $G$ then its convolution is the function $f_1*f_2(h)$ on $G$ defined by
\begin{equation}
\label{conv}
f_1*f_2(h)=\frac{1}{p-1}\sum_{u\in G}f_1(u)f_2(u^{-1}h).
\end{equation}

 \begin{thm}
 \label{nonjacob}
Suppose that $(X,\lambda)$ is  the jacobian of a smooth projective irreducible
genus $g$ curve $\cV$ with canonical principal
 polarization. Then there exists 
 a nonnegative integer-valued function
 $$\mathbf{b}: G=(\bbZ/p\bbZ)^{*} \to \bbZ_{+}\subset \bbC$$ such that
 \begin{equation}
 \label{sumC}
 \sum_{h\in G}\mathbf{b}(h)=\frac{2g}{p-1}+2, \quad
 \sum_{h\in G}\mathbf{b}(h) \mathbf{j}(h^{-1}) \in p\bbZ,
 \end{equation}
 \begin{equation}
 \label{bConvj}
 \mathbf{a}(v):=\mathbf{a}_{X,\delta}(v)=\frac{(p-1)}{p} \cdot \mathbf{b}*\mathbf{j}(-v)-1 \ \forall v \in G.
 \end{equation} 
 \end{thm}
 
 \begin{ex}
 \label{ellipticC}
  Let us  describe explicitly the case $g=1$. Then $X$ is an elliptic curve. It follows from \eqref{divide} that
 $(p-1)$ divides $2$ and therefore
 $$p=3, \ p-1=2, \ G=(\bbZ/3\bbZ)^{*}=\{\bar{1}=1\bmod 3, \ \ \bar{2}=2 \bmod 3=-\bar{1}\}.$$
 Then either
 \begin{equation}
 \label{a11}
 \mathbf{a}_{X,\delta}(\bar{1})=1, \quad \mathbf{a}_{X,\delta}(\bar{2})=0
 \end{equation}
 or
 \begin{equation}
 \label{a10}
 \mathbf{a}_{X,\delta}(\bar{1})=0, \quad \mathbf{a}_{X,\delta}(\bar{2})=1.
 \end{equation}
  
  In the case \eqref{a11}, the desired $\mathbf{b}$ is given by the formulas
  $$\mathbf{b}(\bar{1})=0, \quad \mathbf{b}(\bar{2})=3.$$
 In the case \eqref{a10}, the desired $\mathbf{b}$ is given by the formulas
 $$\mathbf{b}(\bar{1})=3, \quad \mathbf{b}(\bar{2})=0.$$
 \end{ex}

 \begin{proof}[Proof of Theorem \ref{nonjacob}]
In light of Example \ref{ellipticC}, we may assume 
  that $g>1$. We may assume that
$(X,\lambda)= (\cJ(\cV),\Theta)$ where $\cV$ is an
irreducible smooth projective genus $g$ curve, $\cJ(\cV)$ its jacobian
with canonical principal polarization $\Theta$, and
 $$\delta \in \Aut(\cJ(\cV), \Theta)$$ satisfies the $p$th cyclotomic equation 
 $$\sum_{i=0}^{p-1} \delta^i=0\in \End(\cJ(\cV)).$$
 This means that
 $\delta^p=1 \in \mathrm{End}(\cJ(\cV))$ 
and  the (sub)set $\cJ(\cV)^{\delta}$ of $\delta$-invariant points of  $\cJ$ is finite. The latter implies that the homomorphism
$1_{\cJ(\cV)}-\delta: \cJ(\cV) \to \cJ(\cV)$ has finite kernel,
hence, is {\sl surjective}.

Identifying the curve $\cV$ with its image $\mathrm{alb}_{P_0}(\cV)\subset \cJ(\cV)$ 
(i.e.,  
$\mathrm{alb}(P_0)$ is the zero $0$ of the group law on $\cJ(\cV)$), 
we have
$$0=\mathrm{alb}_{P_0}(P_0)\in C \subset \cJ(\cV).$$
If $v \in  X=\cJ(\cV)$  then we write $T_v$ for the translation map
$$T_v: X \to X, \quad x \mapsto  x+v.$$

By Torelli Theorem (Weil's variant, see \cite[p. 35, Hauptsatz]{Weil1} and \cite{Weil2}) applied to 
$$\delta^{(p+1)/2}\in \Aut(\cJ(C),\Theta),$$
$\exists  \ \phi_{1/2} \in \Aut(\cV)$, $\epsilon=\pm 1$ and $z \in \cJ(\cV)$ such that
$$\delta^{(p+1)/2}(P)=\epsilon \ \phi_{1/2}(P)+z \ \ \forall P \in C\subset \cJ(\cV).$$
 This implies that
$$
\delta(P)=\delta^{p+1}(P)=\left(\delta^{(p+1)/2}\right)^2(P)=\delta^{(p+1)/2}(\epsilon \ \phi_{1/2}(P)+z )=$$
$$\epsilon \ \delta^{(p+1)/2}(\phi_{1/2}(P))+\delta^{(p+1)/2}(z)=
\epsilon \ (\epsilon \phi_{1/2}^2(P)+z)+\delta^{(p+1)/2}(z)=$$
$$\epsilon^2 \phi_{1/2}^2(P)+\left(\epsilon \ z+\epsilon \delta^{(p+1)/2} z\right)= \phi_{1/2}^2(P)+\left(\epsilon \ z+\epsilon \delta^{(p+1)/2} z\right).
$$ 
It follows that
$\delta(P)=\phi_{1/2}^2(P)-w$ with $w=-(\epsilon \ z+\epsilon \ \delta^{(p+1)/2} z)$. 
Let us put
$$\phi:= \phi_{1/2}^2\in \Aut(C).$$
Then
\begin{equation}
\label{WeilTw}
\phi(P)=\delta(P)+w=T_w \circ \delta(P) \quad \forall P \in \cV \subset \cJ(\cV),
\end{equation}
i.e., the following diagram is commutative. (Here the horisontal arrows are the inclusion map $\mathrm{alb}_{P_0}$.)

$$\ \
\begin{CD}
C @>\subset>>\cJ(C)\\
@V \phi VV     @VV T_w \circ \delta V \\
C @>\subset>>\cJ(C)
\end{CD}.$$

 Choose $u \in \cJ(C)$ such that $v-\delta(v)=(1-\delta)v=w$.  Then
$$T_w \circ \delta =T_v \circ \delta \circ T_v^{-1},$$ i.e., 
$T_w \circ \delta$ and $\delta$ are conjugate in  the group $\tilde{\Aut}(\cJ(\cV))$ of biregular automorphisms of the
 {\sl algebraic variety} $\cJ(\cV)$. In particular,
 $T_w \circ \delta$ is a periodic automorphism of order $p$; hence,
 $\phi^p$ is the {\sl identity automorphism} of $\cV$.
On the other hand, $\phi$ itself is {\sl not} the identity map (see below).

It is well known that 
 the map $\Psi: \Omega^1(\cJ(\cV)) \to \Omega^1(\cV)$ induced by the  inclusion map $\mathrm{alb}_{P_0}:\cV \subset \cJ(\cV)$ 
 is an isomorphism of  $g$-dimensional complex vector spaces. 
 On the other hand,
the linear map  $(T_u)_{\Omega}: \Omega^1(\cJ(\cV)) \to \Omega^1(\cJ(\cV))$
 induced by any translation $T_u: \cJ(\cV)) \to \cJ(\cV))$ is the {\sl identity map} (because all global regular 1-forms
 on an abelian variety are translation-invariant).  Hence,
 the linear maps 
 $\delta_{\Omega}: \Omega^1(\cJ(\cV)) \to \Omega^1(\cJ(\cV))$ and
  $ (T_w\circ \delta)_{\Omega}: \Omega^1(\cJ(\cV)) \to \Omega^1(\cJ(\cV))$  induced by $\delta$ and $T_w\circ \delta$ respectively 
  do coincide. This implies that the following diagram is commutative.
  
  $$\ \  \begin{CD}
\Omega^1(\cJ(\cV))@>\Psi>>\Omega^1(\cV)\\
@V\delta_{\Omega}=(T_w\circ \delta)_{\Omega} VV     @VV\phi_{\Omega} V \\
\Omega^1(\cJ(\cV) )@>\Psi>>\Omega^1(\cV)
\end{CD}.$$
  
  It follows that
  $$\phi_{\Omega} =\Psi \circ \delta_{\Omega}\circ \Psi^{-1}.$$
   In particular,  the linear operators $\phi_{\Omega}$  and  $\delta_{\Omega}$ have the the same spectrum, the same trace, and $\phi_{\Omega}$ is an automorphism of order $p$.
 This implies that
 $\phi$ is {\sl not} the identity map, hence, has order $p$.
The action of $\phi$  on $\mathcal{C}$ gives rise to
the group embedding  
$$\mu_p \hookrightarrow \Aut(\mathcal{C}), \quad \zeta_p \mapsto \phi.$$
Let $P\in \cV$ be a fixed point of $\phi$. Then $\phi$ induces an automorphism of the
corresponding (one-dimensional) tangent space $\mathcal{T}_P(\cV)$ that is multiplication by a complex
number $\epsilon_P$ that is called the {\sl index} of $P$. 
Clearly, $\epsilon_P$ is a $p$th root of unity.


\begin{lem}
\label{nondegenerate}
Every fixed point $P$ of $\phi$ is nondegenerate, i.e., $\epsilon_P\ne 1$.
\end{lem}

\begin{proof}[Proof of Lemma \ref{nondegenerate}]    The result is  well-known. but I failed to find a proper reference (however, see \cite[Lemma 1.2]{ZarhinManinFest}) where the case $p=3$
was proven.)

 Suppose that $\epsilon_P=1$.
 Let $\OC_P$ be the local ring of $\cV$ at $P$ and $\fm_P$ its maximal ideal. We write $\phi_{*}$ for
 the automorphism of $\OC_P$ induced by $\phi$. Clearly,    $\phi_{*}^p$    is the identity
 map. Since $\phi$ is {\sl not} the identity map, there
 are no $\phi_{*}$-invariant local parameters at $P$. Clearly,
  $\phi_{*}(\fm_P)=\fm_P,    \phi_{*}(\fm_P^2)=\fm_P^2$.
  Since      $T_P(\cV)$ is the dual of   $\fm_P/\fm_P^2$ and $c_p=1$, we conclude that
   $\phi_{*}$ induces the identity map on  $\fm_P/\fm_P^2$.
This implies that if $t\in \fm_P$ is a local parameter at $t$ (i.e., its image $\bar{t}$ in
 $\fm_P/\fm_P^2$   is {\sl not} zero) then 
 $$t^{\prime}:=\sum_{k=0}^{p-1}\phi_{*}^k(t) \in \fm_P\subset \OC_P$$
is
 $\phi_{*}$-invariant and its image in  $\fm_P/\fm_P^2$ equals $p\bar{t}\ne 0$. This implies
 that $t^{\prime}$ is a  $\phi_{*}$-invariant local parameter at $P$.      Contradiction.

 \end{proof}

 \begin{cor}
The quotient $\mathcal{D}:=\cV/\mu_p$ is a smooth projective irreducible curve. The map $\cV\to D$ has degree $p$,
 its ramification points are exactly the images of  fixed points of $\phi$
 and all the ramification indices are $p$.
 \end{cor}

 \begin{lem}
 \label{line}
$\mathcal{D}$ is biregularly isomorphic to the projective line.
 \end{lem}

 \begin{proof}[Proof of Lemma \ref{line}]
 Suppose that the genus of $\mathcal{D}$ is positive. Then there is a nonzero $\omega_0 \in \Omega^1(\mathcal{D})$.
 Its inverse image $\omega$ in  $\Omega^1(\cV)$ is a nonzero $\phi_{\Omega}$-invariant regular $1$-form.
 Hence, the spectrum of  $\phi_{\Omega}$ contains $1$. 
 We have seen  that the linear operators $\phi_{\Omega}$  and  $\delta_{\Omega}$ have the same spectrum.
 Hence, the spectrum of $\delta_{\Omega}$ contains $1$, which is not the case.
 The obtained contradiction proves that the genus of $\mathcal{D}$ is $0$.
 \end{proof}

 \begin{cor}
 \label{hur}
The number $F(\phi)$ of fixed points of $\phi$ is $\frac{2g}{p-1}+2$.
 \end{cor}

\begin{proof}[Proof of Corollary \ref{hur}]
Applying the Riemann-Hurwitz formula to $\cV\to \mathcal{D}$, we get
$$2g-2=p\cdot (-2)+ (p-1)\cdot F(\phi).$$
\end{proof}

\begin{lem}
\label{trace}
Let $\tau$ be the trace of $\phi_{\Omega}:\Omega^1(\cV) \to \Omega^1(\cV)$.
Then
$$\tau=\sum_{j=1}^{p-1}a_j \zeta_p^j=\sum_{h\in G}\mathbf{a}(h)\zeta_p^h.$$
\end{lem}

\begin{proof}[Proof of Lemma \ref{trace}]
We have seen  that the linear operators $\phi_{\Omega}$  and  $\delta_{\Omega}$ have the same trace.
Now the very definition of $a_j$'s implies
that the trace of $\delta_{\Omega}$ equals $\sum_{j=1}^{p-1}a_j \zeta_p^j$.
Hence, $\tau=\sum_{j=1}^{p-1}a_j \zeta_p^j$.
\end{proof}

\begin{lem}
\label{maincomp}
Let $\zeta\in \bbC$ be a primitive $p$th root of unity. Then
\begin{equation}
\label{mainF}
\frac{1}{1-\zeta}=-\frac{\sum_{j=1}^{p-1} j \zeta^j}{p}=-\frac{\sum_{h\in G}\mathbf{j}(h)\zeta^h}{p}.
\end{equation}
\end{lem}

\begin{proof}[Proof of Lemma \ref{maincomp}]   We have
$$(1-\zeta)\left(\sum_{j=1}^{p-1} j \zeta^j\right)=\sum_{j=1}^{p-1}\left( j \zeta^j- j \zeta^{j+1}\right)=
\left(\sum_{j=1}^{p-1} \zeta^j\right)-(p-1)\zeta^p=(-1)-(p-1)=-p.$$
\end{proof}

{\bf End of proof of Theorem \ref{nonjacob}}. Let $B$ be the set of fixed points of $\phi$.
 We know that $\#(B)=\frac{2g}{p-1}+2$. By the holomorphic Lefschetz fixed point formula \cite[Th. 2]{AT},
\cite[Ch. 3, Sect. 4]{GH} (see also \cite[Sect. 12.2 and
12.5]{Milnor})
 applied to $\phi$,
 \begin{equation}
 \label{Lefschetz}
 1-\bar{\tau}=\sum_{P\in B}\frac{1}{1-\epsilon_P}
 \end{equation}
 where $\bar{\tau}$ is the complex-conjugate of $\tau$. Recall that every $\epsilon_P$ is a primitive
  $p$th root of unity. Now Theorem \ref{nonjacob} follows readily from the following assertion.
  \end{proof}
  
  \begin{prop}
  \label{hLfp}
  Let us define for each $h \in G$ the nonnegative integer $\mathbf{b}(h)$ as the number of fixed points
  $P\in B\subset \cV(\bbC)$ such that $\epsilon_P=\zeta_p^h$. Then
  \begin{equation}
  \label{sumBh}
  \sum_{h\in G}\mathbf{b}(h)=F(\phi)=\frac{2g}{p-1}+2.
  \end{equation}
  and
  \begin{equation}
  \label{bTOa}
  \mathbf{a}(v)=\frac{(p-1)}{p}\cdot \mathbf{b}*\mathbf{j}(-v) -1 \ \forall  v\in G.
  \end{equation}
  \end{prop}
  In particular,
  $$\bbZ \ni  1+\mathbf{a}(-1 \bmod p)=\frac{1}{p} \left(\sum_{h \in G} \mathbf{b}(h) \mathbf{j}(h^{-1})\right).$$

  \begin{proof}[Proof of Proposition \ref{hLfp}]
  The equality \eqref{sumBh} is obvious. Let us prove \eqref{bTOa}.
  Combining \eqref{Lefschetz} with Lemma \ref{maincomp} (applied to $\zeta=\zeta_p^h$) and  Lemma \ref{trace}, we get
  $$1-\sum_{h\in G}\mathbf{a}(h)\zeta_p^{-h}=\sum_{u\in G}\mathbf{b}(u)\frac{1}{1-\zeta_p^u}=
  \frac{-1}{p}\left(\sum_{u\in G}\mathbf{b}(u)\left(\sum_{h\in G}\mathbf{j}(h)\zeta_p^{hu}\right)\right)=$$
  $$ \frac{-1}{p}\sum_{v\in G}\left(\sum_{u\in G}\mathbf{b}(u)\mathbf{j}(u^{-1}v)\right)\zeta_p^v=
   \frac{-1}{p}\sum_{v\in G} (p-1) \mathbf{b}*\mathbf{j}(v)\zeta_p^v=
   \frac{-(p-1)}{p}\sum_{v\in G}  \mathbf{b}*\mathbf{j}(v)\zeta_p^v
   $$
   (here we use a substitution $v=hu$).
   Taking into account that
   $$0=1+\sum_{j=1}^{p-1}\zeta_p^j=1+\sum_{v\in G}\zeta_p^v,$$
   we obtain
   $$-\left(\sum_{v\in G}\zeta_p^v\right)-\sum_{h\in G}\mathbf{a}(h)\zeta_p^{-h}= -\frac{(p-1)}{p}\sum_{v\in G}\mathbf{b}*\mathbf{j}(v)\zeta_p^v.$$
   Taking into account that the $(p-1)$-element set
   $$\{\zeta_p^j\mid 1\le j\le p-1\}=\{\zeta_p^v\mid v\in G\}$$
   is a basis of the $\bbQ$-vector space $\bbQ(\zeta_p)$, we get
  $1+\mathbf{a}(-v)=(p-1)\mathbf{b}*\mathbf{j}(v)/p$,  i.e.,
  \begin{equation}
  \label{finab}
  \mathbf{a}(v)=\frac{(p-1)}{p}\cdot \mathbf{b}*\mathbf{j}(-v) -1 \ \forall  v\in G.
  \end{equation}
       \end{proof}
       \begin{rem}
       \label{j0}
       Let us consider the function
       \begin{equation}
       \label{j00}
       \mathbf{j}_0:=\mathbf{j}-\frac{p}{2}: G=(\bbZ/p\bbZ)^{*} \to \bbQ,  \ (j\bmod p) \mapsto j-\frac{p}{2} \ \ \text{ where } \ j=1, \dots, p-1.
       \end{equation}
       Then
       \begin{equation}
       \label{oddJ0}
       \mathbf{j}_0(-u)=-\mathbf{j}_0(u) \ \forall u\in G.
       \end{equation}
       If $\mathbf{a},\mathbf{b}: G \to \bbZ_{+}$ are functions related by \eqref{finab} then
       $$\mathbf{b}*\mathbf{j}(v)=\mathbf{b}*\mathbf{j}_0(v)+\frac{p}{2(p-1)} \sum_{h\in G}\mathbf{b}(h)=
       \mathbf{b}*\mathbf{j}_0(v)+\frac{p}{2(p-1)}\left(\frac{2g}{p-1}+2\right).$$
       This implies that
       $$\frac{(p-1)}{p}\cdot \mathbf{b}*\mathbf{j}(v)=\frac{(p-1)}{p}\cdot \mathbf{b}*\mathbf{j}_0(v)+\frac{g}{p-1}+1$$
       and therefore
       \begin{equation}
       \label{a0}
       \mathbf{a}(v)=\frac{(p-1)}{p}\cdot \mathbf{b}*\mathbf{j}_0(-v) +\frac{g}{p-1} \ \forall  v\in G.
       \end{equation}
       On the other hand, it follows from \eqref{oddJ0} that the convolution  $\mathbf{b}*\mathbf{j}_0$ also satisfies 
       $$\mathbf{b}*\mathbf{j}_0(-v)=\mathbf{b}*\mathbf{j}_0(v) \  \ \forall  v\in G.$$
       This implies that
        $$\mathbf{a}(v)+ \mathbf{a}(-v) = \frac{(p-1)}{p}\cdot \mathbf{b}*\mathbf{j}_0(-v) +\frac{g}{p-1}  +\frac{(p-1)}{p}\cdot \mathbf{b}*\mathbf{j}_0(v) +\frac{g}{p-1} 
       =\frac{2g}{p-1} 
       \ \forall  v\in G.$$
       This implies that
       \begin{equation}
       \label{oddA}
       \mathbf{a}(v)+ \mathbf{a}(-v) =\frac{2g}{p-1}.
       \end{equation}
       (Actually, we already know it for admissible $\mathbf{a}$, see \eqref{Hsymmetry}.)
       It follows from \eqref{oddA} that
       \begin{equation}
       \label{oddAA}
        \mathbf{a}(v)=\frac{2g}{p-1}-\frac{(p-1)}{p}\cdot \mathbf{b}*\mathbf{j}(v) +1 \ \forall  v\in G.
       \end{equation}
       \end{rem}
       
       \begin{cor}
       \label{uniqueOdd}
       We keep the notation and assumptions of Theorem \ref{nonjacob}. Let
       $\mathbf{b}^{\prime}: G \to \bbC$ be a complex-valued function on $G$.
       \begin{itemize}
       \item[(a)]
       The following two conditions are equivalent.
        \begin{enumerate}
        \item[(a1)]
       $\mathbf{a}(v)=\frac{(p-1)}{p} \cdot \mathbf{b}^{\prime}*\mathbf{j}(-v)-1
       \quad \forall v\in G$.
       \item[(a2)]
       The odd parts of functions $\mathbf{b}$ and $\mathbf{b}^{\prime}$   coincide, i.e.,
       \begin{equation}
       \label{oddBprimeB}
       \mathbf{b}^{\prime}(v)-\mathbf{b}^{\prime}(-v)=\mathbf{b}(v)-\mathbf{b}(-v) \quad \forall v\in G;
       \end{equation}
       in addition,
       \begin{equation}
 \label{sumCprime}
\sum_{h\in G}\mathbf{b}^{\prime}(h)= \sum_{h\in G}\mathbf{b}(h)=\frac{2g}{p-1}+2.
 \end{equation}
 \end{enumerate}
 \item[(b)]
       If $p=3$ and  condition (a1)  holds then
       $$\mathbf{b}^{\prime}(v)=\mathbf{b}(v) \ \forall v \in G.$$
       \end{itemize}
       \end{cor}
       
       \begin{proof}
       If $f: G \to \bbC$ is a complex-valued function on $G$ and 
       $\chi: G \to \bbC^{*}$ is a character (group homomorphism) then we write
       $$c_{\chi}(f)=\frac{1}{p-1}\sum_{h\in G}f(h)\bar{\chi}(h)$$
       for the corresponding {\sl Fourier coefficient} of $f$.  For example, if $\chi_0 \equiv 1$ is the  {\sl trivial character} of $G$ then
       $$c_{\chi_0}(f)=\frac{1}{p-1}\sum_{h\in G}f(h).$$
       In particular,
       $$c_{\chi_0}(\mathbf{j})=\frac{1}{p-1}\sum_{j=1}^{p-1} j=\frac{p}{2} \ne 0.$$
       We have
       \begin{equation}
       \label{Forier}
       f(v)=\sum_{\chi \in \hat{G}} c_{\chi}(f)\chi(v) \ \text{ where } \hat{G}=\Hom(G,\bbC^{*}).
       \end{equation}
     
      Let us consider the function
         $$d: G \to \bbC, \ d(v)=\mathbf{b}^{\prime}(v)-\mathbf{b}(v).$$
         Suppose that (a1) holds. 
         We need to check  that (a2) holds, i.e., 
          $$\sum_{h \in G} d(h)=0, \quad d(v)= d(-v) \ \forall v \in G,$$
          which means that 
          $$c_{\chi_0}(d)=0$$
          and 
          for all {\sl odd characters} $\chi$ (i.e., characters $\chi$ of $G$ with
          $$ \chi(-1\bmod p)=-1)$$
          the corresponding {\sl Fourier coefficient}
          $$c_{\chi}(d)=0.$$
         It follows from \eqref{bConvj} that
        $d*\mathbf{j}(-v)=0$ for all $v \in G$, i.e.,
        $$d*\mathbf{j}(v)=0 \ \forall v \in G.$$
        This implies that
        $$0=c_{\chi}(d*\mathbf{j})=c_{\chi}(d)\cdot c_{\chi}(\mathbf{j}) \quad \forall \chi \in \hat{G}.$$
        However, we know that $c_{\chi_0}(\mathbf{j})\ne 0$. On the other hand,
        $c_{\chi}(\mathbf{j}) \ne 0$ for all {\sl odd} $\chi$: 
         it follows from \cite[Chap. 16, Theorem 2]{IR} combined with \cite[Ch. 9,  p. 288, Th. 9.9]{MW}.
         This implies that $c_{\chi_0}(d)=0$ and
        $c_{\chi}(d)=0$ for all {\sl odd} $\chi$. 
        This proves that (a1) implies (a2).
        
        Assume now that (a2) holds. This means that $d(v)$ is an {\sl even} function, i.e.,
        $$d(-v)=d(v) \quad \forall v \in G,$$
        and
        $$\sum_{v \in G}d(v) =0.$$
        We need to prove that (a1) holds. 
        Let us prove first that
        \begin{equation}
       \label{a0P}
       \mathbf{a}(v)=\frac{(p-1)}{p}\cdot \mathbf{b}^{\prime}*\mathbf{j}_0(-v) +\frac{g}{p-1} \ \forall  v\in G.
       \end{equation}
        
         In light of \eqref{a0},  in order to prove \eqref{a0P},  it suffices to check that
        $$d*\mathbf{j}_0(-v) =0 \quad \forall v \in G,$$
        i.e.,
        \begin{equation}
        \label{fj0}
       D_v:= \sum_{h \in G} d(h)\cdot \mathbf{j}_0(h^{-1}v) =0 \quad \forall v \in G.
        \end{equation}
        In order to prove \eqref{fj0}, recall that $\mathbf{j}_0$ is {\sl odd} and $d$ is {\sl even}.
        This implies that
        $$D_v= \sum_{h \in G} d(h)\cdot \mathbf{j}_0(h^{-1}v)= 
        \sum_{h \in G} d(-h)\cdot \mathbf{j}_0((-h)^{-1}v)=$$
        $$\sum_{h \in G} d(h)\cdot \mathbf{j}_0(-h^{-1}v)=
        \sum_{h \in G} d(h)\cdot \left(-\mathbf{j}_0(h^{-1}v)\right)=
        -\sum_{h \in G} d(h)\cdot \mathbf{j}_0(h^{-1}v)=-D_v.$$
        It follows that 
        $$\sum_{h \in G} d(h)\cdot \mathbf{j}_0(h^{-1}v)=D_v=0,$$
         which proves \eqref{a0P}.
        
        Now taking into account that $\mathbf{j}_0=\mathbf{j}-p/2$, we get from  \eqref{a0P} that
        $$\mathbf{a}(v)=\frac{1}{p}\sum_{h \in G} \mathbf{b}^{\prime}(h)\cdot \Big(\mathbf{j}(h^{-1}(-v))-p/2\Big)=$$
        $$\left(\frac{1}{p}\sum_{h \in G} \mathbf{b}^{\prime}(h)\cdot \mathbf{j}(h^{-1}(-v)\right)
        -\frac{1}{2}\left(\sum_{h \in G} \mathbf{b}^{\prime}(h)\right)+\frac{g}{p-1}=$$
        $$\frac{(p-1)}{p}\cdot 
        \mathbf{b}^{\prime}*\mathbf{j}(-v)-\frac{1}{2}\left(\frac{2g}{p-1}+2\right)+\frac{g}{p-1}=
         \mathbf{b}^{\prime}*\mathbf{j}(-v)-1.$$
        So, $\mathbf{a}(v)=\mathbf{b}^{\prime}*\mathbf{j}(-v)-1$, i.e.,
        (a1) holds.
    This ends the proof of (a).
        
        Now let $p=3$. Then $2+2g/(p-1)=g+2$ and $G=\{\bar{1}=1\bmod 3,-\bar{1}\}$.  We already know that
        $\mathbf{b}^{\prime}(\bar{1})-\mathbf{b}^{\prime}(-\bar{1})=\mathbf{b}(\bar{1})-\mathbf{b}(-\bar{1}),$
        $$\mathbf{b}^{\prime}(\bar{1})+\mathbf{b}^{\prime}(-\bar{1})= g+2=\mathbf{b}(\bar{1})+\mathbf{b}(-\bar{1}).$$
        This implies that
        $\mathbf{b}^{\prime}(\bar{1})=\mathbf{b}(\bar{1}), \quad \mathbf{b}^{\prime}(-\bar{1})=\mathbf{b}(-\bar{1}),$
        which proves (b).
       \end{proof}

\begin{rem}
\label{character}
If $v \in G$ then the positive integer $k_v:=\mathbf{j}(h)$ does {\sl not} divide $p$  and
$\mathbf{j}(vh)-k_v \mathbf{j}(h)$ is {\sl divisible} by $p$ for all $h \in G$.
Indeed, by definition of $\mathbf{j}$,
$$v=k_v \bmod p, \quad h=\mathbf{j}(h) \bmod p \in (\bbZ/p\bbZ)^{*}=G.$$
This implies that in $(\bbZ/p\bbZ)^{*}$  we have
$$\left(k_v \mathbf{j}(h)\right) \bmod p=\left(k_v \bmod p\right) \left(\mathbf{j}(h) \bmod p\right)=vh=\mathbf{j}(vh) \bmod p.$$
\end{rem}

\begin{cor}
\label{cPdiv}
Let $c: G \to \bbZ$ be an integer-valued function.  Then the following conditions are equivalent.

\begin{itemize}
\item[(i)]
 $(p-1)\cdot c*\mathbf{j}(1\bmod p)=\sum_{h\in G}c(h)\mathbf{j}(h^{-1})\in p\bbZ$.
\item[(ii)]
 $(p-1)\cdot c*\mathbf{j}(v)=\sum_{h\in G}c(h)\mathbf{j}(h^{-1}v)\in p\bbZ \ \forall v\in G.$
 \end{itemize}
\end{cor}

\begin{proof}
Clearly, (ii) implies (i).  Suppose that (i) holds, i.e.,
$$\sum_{h\in G}c(h)\mathbf{j}(h^{-1})\in p\bbZ.$$
 In order to prove that (ii) holds, we need to check that
$$ \sum_{h\in G}c(h)\mathbf{j}(h^{-1}v)\in p\bbZ \ \forall v\in G.$$
Notice that in light of Remark \ref{character} (applied to $h^{-1}$), if $v \in G$ then there exists $k_v\in \bbZ$
such that $\mathbf{j}(v h^{-1})-k_v \mathbf{j}(h^{-1})$ is {\sl divisible} by $p$ for all $h \in G$. In other words,
$\mathbf{j}(v h^{-1}) \equiv k_v \mathbf{j}(h) \bmod p$. This implies that $\forall v \in G$
$$\sum_{h\in G}c(h)\mathbf{j}(h^{-1}v)=
\sum_{h\in G}c(h)\mathbf{j}(v h^{-1})  \equiv k_v \sum_{h\in G}c(h)\mathbf{j}(h^{-1})\bmod p \equiv 0  \bmod p .$$
\end{proof}

The next assertion shows that not every well rounded function is strongly admissible.

\begin{thm}
\label{bounds}
Suppose that $(X,\lambda)$ is  the jacobian of a smooth projective irreducible
genus $g$ curve $\cV$ with canonical principal
 polarization, and $\delta$ is a periodic automorphism of $(X,\lambda)$ that satisfies the $p$th cyclotomic equation.
 Let $$\mathbf{a}=\mathbf{a}_{X,\delta}: G=(\bbZ/p\bbZ)^{*} \to \bbZ_{+}$$ be the corresponding multiplicity function.
  Then 
\begin{equation}
\label{boundA}
\frac{1}{p}\cdot \frac{2g}{(p-1)}-\frac{p-2}{p} \le \mathbf{a}(v) \le 
\frac{2g}{(p-1)}-\left(\frac{1}{p}\cdot \frac{2g}{(p-1)}-\frac{p-2}{p}\right)
\quad \forall v \in G.
\end{equation}
In particular, if $(p-2) < 2g/(p-1)$ then
$$1 \le  \mathbf{a}(v) \le \frac{2g}{p-1}-1 \quad \forall v \in G.$$
\end{thm}

\begin{proof}
By Theorem \ref{nonjacob},
there exists a {\sl nonnegative} integer-valued function $\mathbf{b}: G \to \bbZ_{+}$ such that
$$\mathbf{a}(h)=  
 \frac{\sum_{u\in G}\mathbf{b}(u)\mathbf{j}(-u^{-1}h)}{p}-1 \ \forall h \in G.$$
Recall that
 $$\mathbf{b}(u)\ge 0, \ \  \sum_{u\in G}\mathbf{b}(u)=\frac{2g}{p-1}+2, \  \ 1 \le \mathbf{j}(v)\le p-1.$$
 This implies that
$$\frac{\sum_{u\in G}\mathbf{b}(u)}{p} -1 \le \mathbf{a}(h)\le (p-1) \cdot \frac{\sum_{u\in G}\mathbf{b}(u)}{p} -1.$$  
This means that
$$\frac{\frac{2g}{p-1}+2}{p} -1 \le \mathbf{a}(h)\le (p-1) \cdot \frac{\frac{2g}{p-1}+2}{p} -1.$$ 
Hence,
$$\frac{1}{p}\cdot \frac{2g}{(p-1)}-\frac{p-2}{p} \le \mathbf{a}(h) \le 
\frac{2g}{(p-1)}-\left(\frac{1}{p}\cdot \frac{2g}{(p-1)}-\frac{p-2}{p}\right)
\quad \forall v \in G.$$
\end{proof}

\section{A construction of jacobians}
\label{construction}

The following theorem may be viewed as an inverse of Theorem \ref{nonjacob}.

\begin{thm}
\label{inverse}
Let $g$ be a positive integer, $p$ an odd prime, $\zeta_p\in \bbC$ a primitive $p$th root of unity, and $G=(\bbZ/p\bbZ)^{*}$.
Suppose that $(p-1)$ divides $2g$.
Let $\mathbf{b}: G \to \bbZ_{+}$ be a nonnegative integer-valued function such that
\begin{itemize}
\item[(i)]
\begin{equation}
\label{sumInverse}
 \sum_{h\in G}\mathbf{b}(h)=\frac{2g}{p-1}+2.
 \end{equation}
 \item[(ii)]
\begin{equation}
\label{inverseB}
(p-1) \cdot \mathbf{b}*\mathbf{j}(1\bmod p)=\sum_{h\in G}\mathbf{b}(h)\mathbf{j}(h^{-1})\in p\bbZ.
\end{equation}
\end{itemize}
Let $\{f_h(x) \mid h \in G\}$ be a $(p-1)$-element set of mutually prime nonzero polynomials $f_h(x)\in \bbC[x]$
that enjoy the following properties.

\begin{itemize}
\item[(1)]
$\deg(f_h)=\mathbf{b}(h)$ for all $h \in G$. In particular, $f_h(x)$ is a (nonzero) constant polynomial if and only
if $\mathbf{b}(h)=0$.
\item[(2)]
Each $f_h(x)$ has no repeated roots.
\end{itemize}

Let us consider the polynomial 
$$f(x)=f_{\mathbf{b}}(x)=\prod_{h\in G}f_h(x)^{\mathbf{j}(h^{-1})} \in \bbC[x]$$
of degree $\sum_{h\in G}\mathbf{b}(h)\mathbf{j}(h^{-1})$.
Let $\cV$ be the smooth projective model of the irreducible  plane affine curve
\begin{equation}
\label{superelliptic} 
y^p=f_{\mathbf{b}}(x)
\end{equation}
endowed with an automorphism
$\delta_{\cV}: \cV \to \cV$ induced by
$$(x,y) \mapsto (x,\zeta_p y).$$
Let $(\mathcal{J},\lambda)$ be the canonically principally polarized jacobian of $\cV$ endowed by the automorphism
$\delta$ induced by $\delta_{\cV}$.  Then $\mathcal{J}$ and $\delta$ enjoy the following properties.

\begin{itemize}
\item[(a)]
$\dim(\mathcal{J})=g$ and
$\sum_{j=0}^{p-1}\delta^j=0$ in $\End(\mathcal{J})$.
\item[(b)] 
Let $\mathbf{a}=\mathbf{a}_{\mathcal{J},\delta}: G \to \bbZ_{+}$ be the corresponding multiplicity function 
defined in \eqref{functionA}. 
Then for all $v \in G$
\begin{equation}
\label{aJ}
\mathbf{a}_{\mathcal{J},\delta}(v)=\frac{(p-1)}{p} \cdot \mathbf{b}*\mathbf{j}(-v)-1 =
\frac{(p-1)}{p}\cdot \mathbf{b}*\mathbf{j}_0(-v) +\frac{g}{p-1} =
\end{equation}
$$\frac{2g}{p-1}-\frac{(p-1)}{p}\cdot \mathbf{b}*\mathbf{j}(v) +1.$$
\end{itemize}
\end{thm}


\begin{proof}[Proof of Theorem \ref{inverse}]
If $\alpha$ is a root of $f(x)$ then there is exactly one $h\in G$ such that $\alpha$ is a root of $f_h(x)$; in addition, the multiplicity of $\alpha$
(viewed as a root of $f(x)$) is $\mathbf{j}(h^{-1})$, which is {\sl not} divisible by $p$. This implies that $f(x)$ is {\sl not} a $p$th power
in the polynomial ring $\bbC[x]$ and even in the field of rational function $\bbC(x)$. It follows from theorem 9.1 of \cite[Ch. VI, Sect. 9]{Lang} that
the polynomial $$y^p-f(x) \in \bbC(x)[y]$$ is irreducible over $\bbC(x)$.  This implies  that the polynomial in two variable
$$y^p-f(x) \in \bbC[x,y]$$ is irreducible, because every its divisor that is a polynomial in $x$ is a constant,  i.e., the affine plane curve \eqref{superelliptic}  is {\sl irreducible}
and its field of rational functions $K$ is the field of fractions of the domain 
$$A=\bbC[x,y]/(y^p-f(x))\bbC[x,y].$$ 
Let $\cV$ be the smooth projective model of the curve \eqref{superelliptic}.
Then $K$ is the field $\bbC(\cV)$ of rational functions on $\cV$; in particular,  $\bbC(\cV)$ is generated over $\bbC$  by rational functions $x,y$. Let 
$\pi: \cV \to \mathbb{P}^1$ be the regular map defined by rational function $x$. Clearly, it has degree $p$. Since 
$$\deg(\pi)=\deg(f)=\sum_{h\in G}\mathbf{b}(h)\mathbf{j}(h^{-1})$$
is divisible by $p$, the map $\pi$ is unramified at $\infty$ (see \cite[Sect. 4]{PS}) and therefore
the set of branch points of $\pi$ coincides with the set of roots of $f(x)$, which, in turn, is the disjoint union
of the sets $R_h$ of roots of $f_h(x)$. In particular, the number of branch points of $\pi$ is
$$\sum_{h\in G}\deg(f_h)=\sum_{h\in G}\mathbf{b}(h)=\frac{2g}{p-1}+2.$$ 
Clearly, $\pi$ is a Galois cover of degree $p$, i.e., the field extension
$$\bbC(\cV)/\bbC(\mathbb{P}^1)=\bbC(\cV)/\bbC(x)$$
 is a cyclic field extension of degree $p$. In addition, the cyclic Galois group
$\Gal(\bbC(\cV)/\bbC(\mathbb{P}^1))$ is generated by the automorphism $\delta_{\cV}: \cV \to \cV$ defined by
$$\delta_{\cV}: \cV \to \cV, \ (x,y) \mapsto (x,\zeta_p y).$$
 It follows from the Riemann-Hurwitz formula (see \cite[Sect. 4]{PS}) that the genus of $\cV$ is 
 $$\frac{\left(\left(\frac{2g}{p-1}+2\right)-2\right)\left(p-1\right)}{2}=g.$$
 In addition, the automorphism $\delta$ of the canonically polarized jacobian $(\mathcal{J},\lambda)$ induced by $\delta_{\cV}$ satisfies the $p$th cyclotomic equation
 $$\sum_{j=0}^{p-1}\delta^j=0 \in \End(\mathcal{J})$$
(see \cite[p. 149]{PS}). Let $B\subset \cV(\bbC)$ be the set of ramification points of $\pi$. Clearly, $B$ coincides with the set of  fixed points of $\delta_{\cV}$.
The map $x: \cV(\bbC) \to \mathbb{P}^1(\bbC)$ establishes a bijection between $B$ and the disjoint union of all $R_h$'s. Let us put
$$B_h=\{P\in B\mid x(P) \in R_h\}.$$
Then $B$ partitions onto a disjoint union of all $B_h$'s and
$$\#(B_h)=\deg(f_h)=\mathbf{b}(h) \ \forall h \in G.$$
Let $P \in   B$. The action of $\delta$ on the tangent space to $C$ at $P$ is multiplication by a certain $p$th root of unity  $\epsilon_P$.
\begin{lem}
\label{index}
$\epsilon_P=\zeta_p^{\mathbf{j}(h)}$ if and only if $P\in B_h$.
\end{lem}
\begin{proof}[Proof of Lemma \ref{index}]
Clearly, if $h_1$ and $h_2$ are {\sl distinct} elements of $G=(\bbZ/p\bbZ)^{*}$ then
$$1 \le \mathbf{j}(h_1), \  \mathbf{j}(h_2) \le p-1; \quad \mathbf{j}(h_1)\ne \mathbf{j}(h_2).$$ 
Therefore $\zeta_p^{\mathbf{j}(h_1)}\ne \zeta_p^{\mathbf{j}(h_2)}$.
Hence, in order to prove our Lemma, it suffices to check that
\begin{equation}
\label{indexP}
\epsilon_P=\zeta_p^{\mathbf{j}(h)} \ \ \text{if } \ P\in B_h.
\end{equation}
So,  let  $P\in B_h$. Then we have
$$x(P)=\alpha\in R_h, \ y(P)=0.$$
Let
$$\mathrm{ord}_P: \bbC(C) \twoheadrightarrow \bbZ$$
be the discrete valuation map attached to $P$. Then one may easily check that
$$\mathrm{ord}_P(x-\alpha)=p, \ \mathrm{ord}_P(x-\beta)=0 \ \forall \beta \in \bbC \setminus
\{\alpha\}.$$
This implies that
$$\mathbf{j}(h^{-1})\cdot \mathrm{ord}_P(x-\alpha)=\mathrm{ord}_P(y^p)=p\cdot \mathrm{ord}_P(y)$$
and therefore
\begin{equation}
\label{ordPy}
\mathrm{ord}_P(y)=\mathbf{j}(h^{-1}).
\end{equation}
In light of \eqref{jmult}, there is an integer $m$ such that
$$\mathbf{j}(h^{-1})\cdot \mathbf{j}(h)=1+pm.$$
Combining this with \eqref{ordPy}, we obtain that
$$\ord_P\left(\frac{y^{\mathbf{j}(h)}}{(x-\alpha)^m}\right)=\mathbf{j}(h^{-1}) \cdot \mathbf{j}(h)-pm=1$$
and therefore 
$t:=y^{\mathbf{j}(h)}/(x-\alpha)^m$ is a {\sl local parameter} of $\cV$ at $P$. Clearly, the action of $\delta$ multiplies
$t$ by $\zeta_p ^{\mathbf{j}(h)}$ and therefore $\epsilon_P=\zeta_p^{\mathbf{j}(h)}$, which proves the Lemma.
\end{proof}
{\bf End of Proof of Theorem \ref{inverse}}.
Now the desired result follows from Proposition \ref{hLfp} applied to $X=\mathcal{J}, \phi=\delta$ 
combined with \eqref{a0P}      and \eqref{oddAA}.
\end{proof}

\section{The $p=3$ case}
\label{Peq3}
Throughout this  this section we assume that  $p=3$ (see also \cite{ZarhinManinFest}).
We have seen (Example \ref{p3weak}) that every well rounded function of degree $g$  is  admissible 
and the  number of such functions is $(g+1)$. We have
$$G=(\bbZ/3\bbZ)^{*}=\{\bar{1}=1\bmod 3, \ \bar{2}=2 \bmod 3=-\bar{1}\}.$$
\begin{rem}
\label{aFb}
Let $\bbb: G \to \bbZ_{+}$ be a nonnegative integer valued function such that
$$\bbb(\bar{1})+\bbb(\bar{2})=g+2, \quad \bbb(\bar{1})+2\bbb(\bar{2})\in 3\bbZ.$$
It follows from Theorem \ref{nonjacob} and \eqref{aJ} that both
$$a_1=(g+1)-\frac{ \bbb(\bar{1})+2\bbb(\bar{2})}{3} \ \text{ and } \ \  \quad a_2=(g+1)-\frac{2\bbb(\bar{1})+ \bbb(\bar{2})}{3}$$
are nonnegative integers, and the function
$$\mathcal{F}_{\bbb}: G \to \bbZ_{+}, \  \ \bar{1}\mapsto a_1, \ \ \bar{2}=-\bar{1}\mapsto a_2$$
is strongly admissible.

\end{rem}

\begin{thm}
\label{MainP3}
Let ${\mathbf a}: G \to \bbZ_{+}$ be a nonnegative integer valued function such that
\begin{equation}
\label{wellR3}
{\mathbf a}(\bar{1})+{\mathbf a}(\bar{2})=g,
\end{equation}
i.e., ${\mathbf a}$ is  well rounded.

Then ${\mathbf a}$ is strongly admissible if and only if
\begin{equation}
\label{MainEq3}
\frac{g-1}{3} \le {\mathbf a}(\bar{1}), \ {\mathbf a}(\bar{2}) \le \frac{2g+1}{3}.
\end{equation}
\end{thm}
\begin{rem}
\label{countA3}
Clearly, if ${\mathbf a}$ is well rounded then ${\mathbf a}(\bar{1})$ satisfies the inequalities \eqref{MainEq3}
if and only if ${\mathbf a}(\bar{2})$ satisfies them.  
This implies that the number of strongly admissible functions of degree $g$
equals the number of integers $a$ such that
\begin{equation}
\label{MainA3}
\frac{g-1}{3} \le a \le \frac{2g+1}{3}.
\end{equation}
Indeed, let us attach to such $a$ the well rounded function ${\mathbf a}: G \to \bbZ_{+}$ defined by
\begin{equation}
\label{a1a23}
{\mathbf a}(\bar{1}):=a \ge \frac{g-1}{3}, \quad {\mathbf a}(\bar{2}):=g-{\mathbf a}(1)=g-a \le g-\frac{g-1}{3}=\frac{2g+1}{3};
\end{equation}
in addition, 
$${\mathbf a}(\bar{2})=g-a \ge g-\frac{2g+1}{3}=\frac{g-1}{3}.$$
By Theorem \ref{MainP3}, ${\mathbf a}$ is strongly admissible. Conversely, every strongly admissible function
${\mathbf a}$ is uniquely determined (as in \eqref{a1a23}) by an integer $a:={\mathbf a}(\bar{1})$ that satisfies the inequalities \eqref{MainA3}.
\end{rem}

\begin{proof}[Proof of  Theorem \ref{MainP3}]
It follows from Theorem \ref{bounds} applied to $p=3$  that every strongly admissible function ${\mathbf a}$
of degree $g$ enjoys properties  \eqref{MainEq3}. 

Conversely, suppose that ${\mathbf a}$ is well rounded
function of degree $g$ that enjoy properties  \eqref{MainEq3}. Let us consider the integers
$$b_1:=(2g+1)-3{\mathbf a}(\bar{1}), \quad b_2:=3 {\mathbf a}(\bar{1})-(g-1).$$
It follows from \eqref{MainEq3} that both $b_1$ and $b_2$ are {\sl nonnegative} integers.
In addition,
$$b_1+b_2=\left((2g+1)-3{\mathbf a}(\bar{1}) \right)+ \left(3 {\mathbf a}(\bar{1})-(g-1)\right)=(2g+1)-(g-1)=g+2;$$
$$b_1+2 b_2=2\left((2g+1)-3{\mathbf a}(\bar{1}) \right)+ \left(3 {\mathbf a}(\bar{1})-(g-1)\right)
=(3g+3)-3 {\mathbf a}(\bar{1}).$$
Hence, $b_1$ and $b_2$ are nonnegative integers such that
$$b_1+b_2=g+2, \quad b_1+2 b_2=(3g+3)-3 {\mathbf a}(\bar{1}) =3 \left(g+1- {\mathbf a}(\bar{1})\right) \in 3\bbZ.$$
It follows from Remark \ref{aFb} that if we consider the function
$$\bbb: G \to \bbZ_{+}, \quad \bar{1} \mapsto b_1, \ \bar{2} \mapsto b_2,$$
then
the function
$$\mathcal{F}_{\bbb}: G \to \bbZ_{+}, \quad 
\bar{1} \mapsto (g+1)-\frac{ \bbb(\bar{1})+2\bbb(\bar{2})}{3},  \ \  \bar{2} \mapsto (g+1)-\frac{2\bbb(\bar{1})+ \bbb(\bar{2})}{3}$$
is strongly admissible; in particular, it is well rounded.  We have
$$\mathcal{F}_{\bbb}(\bar{1})=(g+1)-\frac{ \bbb(\bar{1})+2\bbb(\bar{2})}{3}=(g+1)-\frac{3 \left(g+1- {\mathbf a}(\bar{1})\right)}{3}=
(g+1)-(g+1-{\mathbf a}(\bar{1}))={\mathbf a}(\bar{1}).$$
Since $\mathcal{F}_{\bbb}$ is well rounded,
$$\mathcal{F}_{\bbb}(\bar{2})=\mathcal{F}_{\bbb}(-\bar{1})=g-\mathcal{F}_{\bbb}(\bar{1})=g-{\mathbf a}(\bar{1})={\mathbf a}(-\bar{1})
={\mathbf a}(\bar{2}).$$
 This implies that the function ${\mathbf a}$ coincides with  the strongly admissible function $\mathcal{F}_{\bbb}$ and therefore is strongly admissible itself.
 This ends the proof.
\end{proof}

We finish this section by counting the number $A_3(g)$ of strongly admissible functions of degree $g$, using Remark \ref{countA3}.

\begin{enumerate}
\item[(1)]
If   $g=3k+1$ where $k$ is a nonnegative integer. then  $A_3(3k+1)$ is the number of integers $a$ with
$$k=\frac{g-1}{3} \le a \le \frac{2g+1}{3}=\frac{6k+3}{3}=2k+1.$$
Hence, $A_3(3k+1)=k+2$.
\item[(2)]
If   $g=3k+2$ where $k$ is a nonnegative integer. then  $A_3(3k+2)$ is the number of integers $a$ with
$$k+1/3=\frac{g-1}{3} \le a \le \frac{2g+1}{3}=\frac{6k+4}{3}=2k+1+1/3.$$
Hence, $A_3(3k+2)=k+1$.
\item[(3)]
If   $g=3k$ where $k$ is a positive integer then  $A_3(3k)$ is the number of integers $a$ with
$$k-1/3=\frac{g-1}{3} \le a \le \frac{2g+1}{3}=\frac{6k+1}{3}=2k+1/3.$$
Hence, $A_3(3k)=k+1$.
\end{enumerate}


\section{The CM case}
\label{CMcase}
We use the notation and assumptions of Section \ref{s1}.
Suppose that $\dim(X)=g=(p-1)/2$, i.e. $2g=(p-1)$. Then $X$ becomes an abelian variety of CM type with multiplication by the CM field $\bbQ(\zeta_p)$ of degree $(p-1)$. The corresponding nonnegative  multiplicity function $\mathbf{a}_{X,\delta}$ enjoys the property
$$\mathbf{a}_{X,\delta}(h)+\mathbf{a}_{X,\delta}(-h)=\frac{2g}{p-1}=1,$$
which means that for each $h \in G$ either
$$\mathbf{a}_{X,\delta}(h)=1, \quad \mathbf{a}_{X,\delta}(-h)=0$$
or
$$\mathbf{a}_{X,\delta}(h)=0, \quad \mathbf{a}_{X,\delta}(-h)=1.$$
To each $h \in G=(\bbZ/p\bbZ)^{*}$ corresponds the field embedding
$$\psi_h: \bbQ(\zeta_p) \to \bbC, \ \zeta_p \mapsto \zeta_p^h.$$
Clearly, the CM type of $X$ is the $(p-1)/2$-element set
$$\Psi=\Psi_{X,\delta}=\{\psi_h \mid \mathbf{a}_{X,\delta}(h)=1\}.$$

\begin{ex}
\label{hyper}
Let $\mathcal{C}$ be the smooth projective model of the plane affine curve $y^2=1-x^p$.
Then $\mathcal{C}$ has genus $g=(p-1)/2$ and admits an automorphism 
$\delta_{\mathcal{C}}: \mathcal{C} \to \mathcal{C}$ induced by
$$(x,y) \mapsto (\zeta_p x,y).$$
Let $(\mathcal{J},\lambda)$  be the canonically principally polarized jacobian of $\mathcal{C}$ endowed by the automorphism $\delta\in \Aut(\mathcal{J},\lambda)$ induced by $\delta_{\mathcal{C}}$.  It is known \cite[Example 15.4 (2)]{ST} that 
$\zeta_p \to \delta$ can be extended to the ring homomorphism
$\bbZ[\zeta_p] \to \End(\mathcal{J})$ (i.e., $\sum_{i=0}^{p-1}\delta^i=0$), which makes $\mathcal{J}$ an abelian variety of CM type with multiplication by $\bbQ(\zeta_p)$ and its CM type
$\Psi$ is $\{\psi_i \mid 1 \le i \le g=(p-1)/2\}$.  This means that the corresponding multiplicity function $\mathbf{a}_{\mathcal{J},\delta}(h)$ is as follows.
$$\mathbf{a}_{\mathcal{J},\delta}(i \bmod p)=1 \ \text{ if }  1 \le i \le (p-1)/2;$$
$$\mathbf{a}_{\mathcal{J},\delta}(i  \bmod p)=0 \ \text{ if }   i \ge (p-1)/2.$$
\end{ex}

Recall that $p$ is an odd prime.
The case $p=3$ (with arbitrary $g$) was discussed in detail in Section \ref{Peq3}. The following assertion deals with $p>3$ when $g=(p-1)/2$.

\begin{thm}
\label{pCM}
Let $p>3$ and $g=(p-1)/2$.  
Then  the number of strongly admissible functions of degree $g$ is $(p^2-1)/6$.
In particular,  every well rounded function of degree $g$ is  strongly admissible if and only if
$p \in \{5,7\}$.
\end{thm}

\begin{proof}
So, we need to compute the number of strongly admissible functions when $2g=p-1$.
By Theorem \ref{nonjacob}, each strongly admissible function is of the form
$\frac{(p-1)}{p} \cdot \mathbf{b}*\mathbf{j}(-v)-1$ 
where the nonnegative integer-valued function $\mathbf{b}: G \to \bbZ_{+}$ enjoys the properties
\begin{equation}
\label{part3}
\sum_{h\in G}\mathbf{b}(h)=\frac{2g}{p-1}+2=1+2=3;
\end{equation}
\begin{equation}
\label{convP}
\sum_{h\in G}\mathbf{b}(h) \mathbf{j}(h^{-1}) \in p\bbZ.
\end{equation}
Taking into account that
$$1 \le  \mathbf{j}(h^{-1}) <p \quad \forall h \in G,$$
we conclude that 
$$\sum_{h\in G}\mathbf{b}(h) \mathbf{j}(h^{-1})<\left(\sum_{h\in G}\mathbf{b}(h)\right)p=3p.$$
Hence, \eqref{convP} means that
either
\begin{equation}
\label{SUMp}
\sum_{h\in G}\mathbf{b}(h) \mathbf{j}(h^{-1}) = p 
\end{equation}
or
\begin{equation}
\label{SUM2p}
\sum_{h\in G}\mathbf{b}(h) \mathbf{j}(h^{-1}) =  2p.
\end{equation}
Let us compute the number of functions $\mathbf{b}$ that enjoy either properties
\eqref{part3} and \eqref{SUMp} or properties
\eqref{part3} and \eqref{SUM2p}.
First, let us prove that
\begin{equation}
\label{b012}
\mathbf{b}(h) \in \{0,1,2\} \ \forall h \in G.
\end{equation}
Indeed, it follows readily from \eqref{part3} that
$$\mathbf{b}(h) \in \{0,1,2,3\} \ \forall h \in G.$$
Suppose that $\mathbf{b}(v)=3$ for some $v \in G$. Then it follows from \eqref{part3} that all other values of $\mathbf{b}$ are zeros. Now \eqref{convP} implies that 
$$3\cdot  \mathbf{j}(v^{-1})=\mathbf{b}(v) \mathbf{j}(v^{-1})=\sum_{h\in G}\mathbf{b}(h) \mathbf{j}(h^{-1})  \in p\bbZ.$$
Since   a positive integer $\mathbf{j}(v^{-1})$ is strictly less than $p$,  the prime $p$ must divide $3$, which is not true, because $p  >3$. The obtained contradiction proves 
\eqref{b012}.

Now notice that
\begin{equation}
\label{jMINj}
\mathbf{j}(v)+\mathbf{j}(-v)=p \ \forall v \in G
\end{equation}
(it follows readily from the very definition of $\mathbf{j}$).  Let us consider the function
\begin{equation}
\label{barB}
\bar{\mathbf{b}}: G \to \bbZ_{+}, \ h\mapsto \mathbf{b}(-h).
\end{equation}
Clearly,
$$\sum_{h \in G}\bar{\mathbf{b}}(h)=\sum_{h \in G}\mathbf{b}(h)=3; \quad
\bar{\mathbf{b}}(G)=\mathbf{b}(G) \subset \{0,1,2\}.$$
On the other hand,
$$\sum_{h\in G}\bar{\mathbf{b}}(h) \mathbf{j}(h^{-1}) =
\sum_{h\in G}\mathbf{b}(-h) \mathbf{j}(h^{-1})=
\sum_{h\in G}\mathbf{b}(h) \mathbf{j}(-h^{-1})=
\sum_{h\in G}\mathbf{b}(h) \left(p-\mathbf{j}(h^{-1})\right)=$$
$$p \left(\sum_{h\in G}\mathbf{b}(h) \right)-\left(\sum_{h\in G}\mathbf{b}(h) \mathbf{j}(h^{-1})\right)=3p-\left(\sum_{h\in G}\mathbf{b}(h) \mathbf{j}(h^{-1})\right).$$
It follows that 
$$\sum_{h\in G}\bar{\mathbf{b}}(h) \mathbf{j}(h^{-1}) =3p-p=2p$$
if $\mathbf{b}$ satisfies \eqref{SUMp}. On the other hand, if 
$\mathbf{b}$ satisfies \eqref{SUM2p} then
$$\sum_{h\in G}\bar{\mathbf{b}}(h) \mathbf{j}(h^{-1}) =3p-2p=p.$$
It follows from Theorem \ref{inverse} combined with Remark \ref{round}(iii) that
$$h \mapsto \bar{\mathbf{a}}(h)=\mathbf{a}(-h)=\frac{2g}{p-1}-\mathbf{a}(h)
=1-\mathbf{a}(h)$$
is a strongly admissible function of degree $g=(p-1)/2$. This implies that  $\mathbf{a}(1 \bmod p)=1$ (resp. $0$) if and only if
$\bar{\mathbf{a}}(1 \bmod p) =0$ (resp. $1$).

Notice that
$$\mathbf{a}(-1 \bmod p) =\frac{\sum_{h \in G}\mathbf{b}(h) \mathbf{j}(h^{-1})}{p}-1.$$
This implies that $\mathbf{a}(1 \bmod p) =1$ (resp. $0$) if and only if
$\sum_{h \in G}\mathbf{b}(h) \mathbf{j}(h^{-1})=p$ (resp. $2p$).


We will need the following two auxiliary assertions.
\begin{lem}
\label{partitions}
Let $Q_3(p)$  be the set of partitions in 3 parts of $p$.

There is a natural bijection between $Q_3(p)$ and the set of strongly admissible functions
$\mathbf{a}: G \to \bbZ_{+}$ of degree $g=(p-1)/2$ such that
$$\mathbf{a}(1 \bmod p)=1.$$
\end{lem}

\begin{lem}
\label{partitions3}

The cardinality of $Q_3(p)$ is $(p^2-1)/12$.
\end{lem}

{\bf End  of proof of Theorem \ref{pCM} (modulo Lemmas  \ref{partitions} and \ref{partitions3})}.  
The map $\mathbf{a} \mapsto \bar{\mathbf{a}}$ is a bijection between the sets of strongly admissible functions that take on at $1\bmod p$ the the values $0$ and $1$ respectively. Applying both lemmas, we conclude that the number of all strongly admissible functions functions of degree $(p-1)/2$ is twice the cardinality of $Q_3(p)$, i.e., this number is
$$2 \cdot \frac{p^2-1}{12}=\frac{p^2-1}{6}.$$
which  proves first assertion of Theorem \ref{pCM}. As for the second one, recall that the number of well rounded functions of degree $(p-1)/2$ is $2^{(p-1)/2}$. It remains to notice that $2^{(p-1)/2}>(p^2-1)/6$ if $p \ge 11$ while
$2^{(p-1)/2}=(p^2-1)/6$ if $p \in \{5,7\}$.
\end{proof}

\begin{proof}[Proof of Lemma \ref{partitions}]
If $v \in G$ then we write ${\mathbf{\delta}}_v: G \to \bbZ_{+}$ for the corresponding {\sl delta function} that takes on value $1$ at $v$ and vanishes elsewhere. 

Each element $M$ of $Q_3(p)$ may be viewed as a non-ordered triple $\{m_1,m_2,m_3\}$ of positive integers, whose sum $\sum_{i=1}^3 m_i=p$. Clearly, 
$$p > m_i \ne p-m_j \quad \forall i,j \in \{1,2,3\}.$$
 Let us define
$$h_i:=(m_i \bmod p)^{-1} \in (\bbZ/p\bbZ)^{*}=G.$$
Clearly,
\begin{equation}
\label{ijminus}
h_i^{-1}=m_i \bmod p, \ \  \mathbf{j}(h_i^{-1})=m_i, \ \  
h_i^{-1} \ne - h_j^{-1} \quad \forall i,j=1,2,3.
\end{equation}
Let us consider the function
$$\mathbf{b}_M=\sum_{i=1}^3 \mathbf{\delta}_{h_i} : G \to \bbZ_{+}.$$
We have
$$\sum_{h \in G}\mathbf{b}_M(h)=\sum_{i=1}^3 
\left(\sum_{h \in G}\mathbf{\delta}_{h_i}(h)\right)=\sum_{i=1}^3 1=3,$$
$$\sum_{h \in G}\mathbf{b}_M(h) \mathbf{j}(h^{-1})=\sum_{i=1}^3 
\left(\sum_{h \in G}\mathbf{\delta}_{h_i}(h)\mathbf{j}(h^{-1})\right)=
\sum_{i=1}^3\mathbf{j}(h_i^{-1})=
\sum_{i=1}^3 m_i=p.$$
By Theorem \ref{inverse},
 the function
$$\mathbf{a}_M: G \to \bbZ_{+}, \ \ v \mapsto \frac{p-1}{p}b_M*\mathbf{j}(-v)-1=
\frac{\sum_{h \in G}b_M(h)\mathbf{j}(-h^{-1}v)}{p}-1$$
is strongly admissible; in addition, 
$$
\mathbf{a}_M(1 \bmod p)=\frac{2g}{p-1}-\mathbf{a}_M(-1 \bmod p)=$$
$$1-\left(\frac{\sum_{h \in G}b_M(h)\mathbf{j}(h^{-1})}{p}-1\right)=
1-\left(\frac{p}{p}-1\right)=
1. $$

Let us consider the map
$$\mathcal{T}: M \to \mathbf{a}_M$$ from $Q_3(p)$ to the set of strongly admissible functions
 $G \to \bbZ_{+}$ of degree $(p-1)/2$  that take on value $1$ at $1 \bmod p$. Let us check that $\mathcal{T}$ is {\sl bijective}. 
 
  In order to check the injectiveness of $\mathcal{T}$, notice that in light of the last inequality of \eqref{ijminus}
 $$\mathbf{b}_M(h)=\max\{\mathbf{b}_M(h)-\mathbf{b}_M(-h), \ 0 \},$$
 which means that  the function $\mathbf{b}_M(h)$ is uniquely determined by its ``odd part''. It follows from Corollary \ref{uniqueOdd} that $\mathcal{T}$ is {\sl injective}.

In order to check that  $\mathcal{T}$ is {\sl surjective}, let us start with a strongly admissible function $\mathbf{a}: G \to \bbZ_{+}$ of degree $(p-1)/2$ with $\mathbf{a}(1\bmod p)=1$.
We  know that there is a function
$$\mathbf{b}: G \to \{0,1,2\} \subset \bbZ_{+}$$
such that
\begin{equation}
\label{sumB}
\sum_{h\in G}\mathbf{b}(h)=3,  \quad \sum_{h \in G}b(h)\mathbf{j}(h^{-1})=p;
\end{equation}
$$\mathbf{a}(v)=
\frac{p-1}{p}\mathbf{b}*\mathbf{j}(-v)-1=
\frac{\sum_{h \in G}b(h)\mathbf{j}(-h^{-1}v)}{p}-1 \ \forall v \in G.$$
We need to find a partition $M$ of $p$ in three parts such that $\mathbf{b}=\mathbf{b}_M$
(which would imply that that $\mathbf{a}=\mathbf{a}_M$). Recall that $\mathbf{b}$ is a nonnnegative integer-valued function.
In light of of first equality of \eqref{sumB}, 
there is a 3-element collection $\{h_1,h_2,h_3\}$ of (not necessarily distinct) elements of $G$ such that
$$\mathbf{b}=
\mathbf{\delta}_{h_2}+\mathbf{\delta}_{h_2}+\mathbf{\delta}_{h_1}.$$
 Let us define the 3-element collection
$$M:=\{m_1=\mathbf{j}(h_1^{-1}),\  m_2=\mathbf{j}(h_2^{-1}),
 \ m_3=\mathbf{j}(h_3^{-1})\}$$
 of (not necessarily distinct) positive integers.
 In light of second equality of \eqref{sumB},
 $$p=\sum_{h \in G}b(h)\mathbf{j}(h^{-1})=\sum_{i=1}^3 \mathbf{j}(h_i^{-1})=
 m_1+m_2+m_3,$$
 i.e., $M$ is a partition of $p$ in three parts.
 Clearly,  we have
 $\mathbf{b}=\mathbf{b}_M$
 and therefore $\mathbf{a}=\mathbf{a}_M$, which proves the surjectiveness of $\mathcal{T}$.
\end{proof}

\begin{proof}[Proof of Lemma \ref{partitions3}]
We know that the prime $p>3$ is congruent to $\pm 1$ modulo $6$. This implies  a well known assertion that that $p^2-1$ is divisible by $12$ (and even by $24$). This implies that 
$(p^2-1)/12$ is the {\sl nearest integer} to $p^2/12$.

It is well known \cite[Sect. 3.1, p. 16]{Andrews} that the cardinality $\#(Q_3(p))$ of
$Q_3(p)$ coincides with the number of partitions of $p-3$ in at most three parts. On the other hand,  it is known 
\cite[Sect. 6.2, p. 58]{Andrews} that the latter number is the nearest integer to
$\frac{((p-3)+3)^2}{12}$, i.e., is the nearest integer to $p^2/12$, which (as we have already seen) is 
$(p^2-1)/12$. This ends the proof of  Lemma \ref{partitions3}.
\end{proof}

\section{Self-products of abelian varieties that are not jacobians}
\label{RosatiS}
\begin{sect}
\label{Rosati}
Let us start by recalling some generalities about endomorphism algebras of abelian varieties and Rosati (anti)involutions
\cite[Sect. 19--21]{MumfordAV}.

Let $X$ be a positive-dimensional abelian variety over an arbitrary algebraically closed field,  $\End(X)$ its endomorphism ring and $\cZZ_X$ the center of of $\End(X)$. Then
$\cZZ_X^{0}=\cZZ_X\otimes \bbQ$ is the center of the  endomorphism algebra $\End^0(X):=\End(X)\otimes \bbQ$; the latter is a finite-dimensional semisimple $\bbQ$-algebra \cite[Sect. 19, Cor. 2]{MumfordAV}. More precisely, 
  there is an isomorphism of $\bbQ$-algebras
$$\End^0(X) \cong \oplus_{i=1}^{\ell} \mathcal{H}_i=: \mathcal{H}$$ 
where $\mathcal{H}_i$ is a central simple algebra over a number field $K_i$ of dimension $d_i^2$  where $d_i$ is a positive integer
while $K_i$ is either totally real or a CM field \cite[Sect. 21, Application I]{MumfordAV}.  In what follows, we will identify $\End^0(X)$
with $\mathcal{H}$ and will  view each  direct summand $\mathcal{H}_i$ as the corresponding two-sided ideal  of $\mathcal{H}$.
Then 
$$\cZZ_X^{0}= \oplus_{i=1}^{\ell} K_i \subset \oplus_{i=1}^{\ell} \mathcal{H}_i=\mathcal{H}.$$
We write $e_i$ for the identity element of $\mathcal{H}_i$, viewed as the certain idempotent of $\End^0(X)$.
Clearly,
$$e_i \in e_i \cZZ_X^{0}= K_i\subset \mathcal{H}_i=e_i  \End^0(X)= \End^0(X) e_i;$$
$$ e_i e_j=0 \ \ \forall i \ne j; \quad \sum_{i=1}^{\ell} e_i=1 \in \mathcal{H}.$$
In addition, $\{e_1, \dots e_{\ell}\}$ is the list of all minimal idempotents in $\cZZ_X^{0}$.

One may view $\End(X)$ and $\cZZ_X$ as orders in $\End^0(X)=\mathcal{H}$ and $\cZZ_X^0$ respectively.

We write
$\mathrm{tr}_{ \mathcal{H}_i/K_i}: \mathcal{H}_i \to K_i$ for the $K_i$-linear {\sl reduced trace} map of the central simple $K_i$-algebra  $\mathcal{H}_i$ over $K_i$
 \cite[Ch. 2, Subsect. 9a]{Reiner}. Recall its definition and basic properties. Let $E_i$ be an overfield of $K_i$ that splits  $\mathcal{H}_i$. There exists  an isomorphism
 $h_i:  \mathcal{H}_i\otimes_{K_i} E_i  \cong \Mat_{d_i}(E_i)$ of central simple $E_i$-algebras
 (where $\Mat_{d_i}(E_i)$ is the algebra of square matrices of size $d_i$ with entries in $E_i$). Then  for all $u_i \in \mathcal{H}_i$, one defines
 $\mathrm{tr}_{ \mathcal{H}_i/K_i}(u_i)$ as the trace of the matrix $h_i(u_i\otimes 1)$; this trace
  lies in $K_i$ and does {\sl not} depend on the choice of $E_i$ and $h_i$. E.g., if $u_i \in K_i\subset \mathcal{H}_{i}$ then
 $h_i(u\otimes 1)$ is the {\sl scalar matrix} $u_i \mathrm{I}_{d_i}$ where $\mathrm{I}_{d_i}$ is the identity matrix in $ \Mat_{d_i}(E_i)$.
 The trace of  the scalar matrix $u_i \mathrm{I}_{d_i}$ is obviously $d_i u_i$, which implies that $\mathrm{tr}_{\mathcal{H}_{i}/K_i}$ coincides with multiplication by $d_i$ on $K_i$. We also have 
 $$\mathrm{tr}_{\mathcal{H}_{i}/K_i}(u_i v_i)=\mathrm{tr}_{\mathcal{H}_{i}/K_i}(v_i u_i) \quad \forall u_i, v_i \in \mathcal{H}_i.$$
We write
$$\mathrm{tr}_{\mathcal{H}/\bbQ}: \mathcal{H}=\oplus_{i=1}^{\ell} \mathcal{H}_i \to \bbQ$$
for the {\sl reduced trace map} on the $\bbQ$-algebra $\mathcal{H}$, which is defined as
$$(u_1, \dots, u_{\ell}) \mapsto \sum_{i=1}^{\ell} \mathrm{Tr}_{K_i/\bbQ}( \mathrm{tr}_{ \mathcal{H}_i/K_i}(u_i)) \quad
\ \text{ where } \  u_i \in \mathcal{H}_i \ \ \forall i,$$
and  $\mathrm{Tr}_{K_i/\bbQ}: K_i \to \bbQ$ is the usual trace map attached to the field extension $K_i/\bbQ$
\cite[Ch. 2, Subsect. 9b]{Reiner}.  Clearly, $\mathrm{tr}_{\mathcal{H}/\bbQ}$ coincides with 
the composition $\mathrm{Tr}_{K_i/\bbQ} \circ \mathrm{tr}_{\mathcal{H}_{i}/K_i}$ on 
$\mathcal{H}_i\subset \mathcal{H}$
and therefore coincides with $d_i  \cdot \mathrm{Tr}_{K_i/\bbQ}$ on $K_i$.
It is also clear that $\mathrm{tr}_{\mathcal{H}/\bbQ}$ is  a $\bbQ$-linear map such that
$$\mathrm{tr}_{\mathcal{H}/\bbQ}(uv)=\mathrm{tr}_{\mathcal{H}/\bbQ}(vu) \quad \forall u,v \in \mathcal{H}.$$

If $\lambda$ is a polarization on $X$ then it gives rise to the so called {\sl Rosati involution} (actually, antiinvolution)
$$\End^0(X) \to \End(X), \ u \mapsto u^{*},  \ (u^*)^{*}=u, \ (uv)^{*}=v^{*}u^{*} \ \forall u, v \in \End^0(X)=\mathcal{H},$$
which is a $\bbQ$-linear map \cite[Sect. 20]{MumfordAV}. The Rosati involution is {\sl positive} \cite[Sect. 21]{MumfordAV}, i.e.,
$$\mathrm{tr}_{\mathcal{H}/\bbQ}(u u^{*})=\mathrm{tr}_{\mathcal{H}/\bbQ}(u^{*}u) >0 \quad \forall \ \text{nonzero} \ u \in \mathcal{H}.$$

Clearly,  this involution  defines an automorphism (an honest involution)
\begin{equation}
 \label{RosatiZ0}
\cZZ_X^0 \to \cZZ_X^0, \ u \to u^{*}
\end{equation}
of the commutative semisimple $\bbQ$-algebra  $\cZZ_X$.  It follows that the Rosati involution permutes the set of minimal idempotents 
$\{e_1, \dots e_{\ell}\}$. On the other hand, if $e_i^{*}=e_j$  with $j \ne i$ then $e_i e_j=0$, hence,
$$\mathrm{tr}_{\mathcal{H}/\bbQ}(e_i e_i^*)=\mathrm{tr}_{\mathcal{H}/\bbQ}(e_i e_j)=\mathrm{tr}_{\mathcal{H}/\bbQ}(0)=0,$$
which contradicts the positivity of the Rosati involution.  Hence, $e_i^{*}=e_i^{*}$ for all $i$. It follows that the Rosati involution
sends $\mathcal{H}_i =e_i \mathcal{H}$ to $\mathcal{H}e_i=\mathcal{H}_i$, i.e., $\mathcal{H}_i$ goes to itself under this involution.
This implies that the center $K_i$ of $\mathcal{H}_i$ goes to itself under this involution. The positiveness of the Rosati involution on $\mathcal{H}$
implies that for all {\sl nonzero} $u \in K_i$
$$d_i \cdot  \mathrm{Tr}_{K_i/\bbQ}(u u^{*})> 0, \ \text{ i.e. } \  \mathrm{Tr}_{K_i/\bbQ}(u u^{*})> 0.$$
It follows that  $u \mapsto u^{*}$ is a positive involution on $K_i$. By Albert's classification \cite[Sect. 21, Application I]{MumfordAV}, this involution acts on $K_i$
 as the identity map if $K_i$ is totally real and  as the complex conjugation if $K_i$  is a CM field. This implies the $\bbQ$-algebra automorphism \eqref{RosatiZ0} of the center does {\sl not} depend on the choice of $\lambda$.


On the other hand, if $u$ is automorphism of $X$ then $u \in \Aut(X,\lambda)$ if and only if $u^{*}u=1_X$ 
\cite[Definition in Sect. 8 and Sect. 21, Proof of Th. 5, first paragraph]{MumfordAV}. 
(Over $\bbC$ this well known assertion follows readily from
the very definition of Rosati involution \cite[Sect. 5.1, p. 114]{BL} combined with the commutative diagram in \cite[Cor. 2.4.6 on p. 36]{BL}.)
It follows that if $u \in \cZZ_X$ respects one polarization on $X$ then it respects all of them!  
\end{sect}
Now let us return to our study of complex abelian varieties with automorphisms of prime period $p$.
\begin{thm}
\label{CentralEx}
Let $p$ be an odd prime, $g$ a positive unteger such that $(p-1)$ divides $2g$,
$X$  a complex $g$-dimensional abelian variety endowed with the ring embeddings
$$\kappa: \bbZ[\zeta_p]\hookrightarrow \cZZ_X \subset \End(X), \ 1 \mapsto 1_X$$
where $\cZZ_X$ is the center of $\End(X)$. Let us put 
$$\delta:=\kappa(\zeta_p) \in  \cZZ_X \subset \End(X).$$
If $X$ is isomorphic as an algebraic variety to the jacobian of a smooth connected projective curve of genus $g$ then it enjoys one of the following two properties.

\begin{itemize}
 \item[(i)]
$$\frac{2g}{p-1} \le p-2.$$
\item[(ii)]
Every  primitive $p$th root of unity $\zeta$ is an eigenvalue of 
$\delta_{\Omega}: \Omega^1(X) \to \Omega^1(X)$ and its multiplicity is 
greater or equal than
$$\frac{1}{p}\cdot \frac{2g}{p-1} - \frac{p-2}{p}.$$
\end{itemize}
\end{thm}

\begin{proof}
 Clearly, $\kappa$ extends by $\bbQ$-linearity to the embedding of $\bbQ$-algebras
 $$\bbQ(\zeta_p)\hookrightarrow \cZZ_X \subset \End^0(X), \quad
 1 \mapsto 1_X, \ \zeta_p \mapsto \delta,$$
 which we continue to denote by $\kappa$.
 Since $\bbQ(\zeta_p)$ is a CM field, the center $\cZZ_X$ is either a CM field or a product of CM fields, and 
 (as we have seen in  Subsection \ref{Rosati}) the Rosati involution coincides with the complex conjugation on each factor. It follows that
 $$\kappa(\bar{u})=\left(\kappa(u)\right)^* \quad \forall u \in \bbQ(\zeta_p).$$
 Taking into account that $\bar{\zeta_p}\zeta_p=1$ (where $\bar{\zeta_p}$ is the complex-conjugate of $\zeta_p$), we conclude that
 $$\delta^* \delta =\kappa\left(\bar{\zeta_p}\right)\kappa(\zeta_p)=
 \kappa\left(\bar{\zeta_p}\zeta_p\right)=\kappa(1)=1_X,$$
 i.e., $\delta^* \delta =1_X$, which means that $\delta \in \Aut(X,\lambda)$ for any polarization $\lambda$ on $X$. Now the desired result follows from 
 Theorem \ref{bounds}.
\end{proof}

\begin{cor}
\label{powersC}
 Let $p$ be an odd prime, $g_0$ a positive unteger such that $(p-1)$ divides $2g_0$.
 Let 
$Y$  be a complex $g_0$-dimensional abelian variety endowed with the ring embeddings
$$\kappa: \bbZ[\zeta_p]\hookrightarrow \cZZ_Y \subset \End(Y), \ 1 \mapsto 1_Y.$$
Let us put 
$$\delta_Y:=\kappa(\zeta_p) \in  \cZZ_Y \subset \End(Y).$$
Suppose that there is a primitive $p$th root of unity $\zeta$ that enjoys one of the following two properties.
\begin{itemize}
 \item[(i)]
 $\zeta$ is not an eigenvalue of $\delta_{Y,\Omega}: \Omega^1(X) \to \Omega^1(X)$.  
 \item[(ii)]
 $\zeta$ is an eigenvalue of $\delta_{Y,\Omega}$ but its multiplicity $a$ is strictly less than
 $$\frac{1}{p} \cdot \frac{2g_0}{p-1}.$$
\end{itemize}

Then the self-product $Y^r$ of $Y$ is not isomorphic as an algebraic variety to the jacobian of a smooth connected projective curve for all positive integers 
$$r >  M:=\frac{p-2}{\frac{1}{p} \frac{2g_0}{p-1}-a}
.$$
(In the case (i) we put $a=0$.)
\end{cor}

\begin{proof}
First, notice that the existence of $\kappa$ implies that the ratio
$$\frac{2g_0}{p-1}=\frac{2\dim(Y)}{p-1}$$
is an {\sl integer}.

Let $r>M$ be a positive integer and 
$X=Y^r$. Then $g:=\dim(Y)=r g_0$, and
the endomorphism ring $\End(X)=\End(Y^r)$ is canonically isomorphic to the matrix ring $\Mat_r(\End(Y))$ of size $r$ over $\End(Y)$  with the same center
as $\End(Y)$. In particular, the image of the {\sl diagonal embedding}
$$\kappa_k: \bbZ[\zeta_p] \hookrightarrow \oplus_{i=1}^k \End(Y)\subset \Mat_r(\End(Y)), \quad u \mapsto (\kappa(u), \dots, \kappa(u)) \ \   (k \text{ times)} $$
lies in the center $\cZZ_Y$ of  $\End(Y)$. On the other hand, 
$$\Omega^1(X)=\Omega^1(Y^r)=\oplus_{i=1}^r \Omega^1(Y)$$
and the linear operator 
$$\delta_{X,\Omega}: \Omega^1(X) \to \Omega^1(X)$$
(which  acts ``diagonally'' on $\Omega^1(X)$)
enjoys the following properties.

Its spectrum coincides with the spectrum of 
$\delta_{Y,\Omega}$. In addition,
 if an eigenvalue $\gamma$ of $\delta_{Y,\Omega}$ has multiplicity $a$ then it has multiplicity $ka$, viewed as an eigenvalue of $\delta_{X,\Omega}$.
 
 Assume now that $X$ is isomorphic to the jacobian of a smooth connected projective curve of genus $g$.  If $\zeta$ is not an eigenvalue of $\delta_{X,\Omega}$ then it is not an eigenvalue of $\delta_{Y,\Omega}$ and therefore
 $$\frac{2r g_0}{p-1}=\frac{2g}{p-1} \le p-2.$$
 This implies that 
 $$r \le \frac{p-2}{\frac{2g_0}{p-1}}=M,$$
 because in this case $a=0$. This contradicts our assumption on $r$.
 
 So,  $\zeta$ is an eigenvalue of $\delta_{Y,\Omega}$ and its multiplicity $a$ satisfies the inequality
 $$a<\frac{1}{p} \cdot \frac{2g_0}{p-1}=    \frac{1}{p} \cdot \frac{2\dim(Y)}{p-1}.$$
 Since both $a$ and $\frac{2g_0}{p-1}$ are integers,
 \begin{equation}
 \label{Ya}
 a\le 
 \frac{1}{p} \cdot \frac{2g_0}{p-1}-\frac{1}{p}.
\end{equation}
This implies that $\zeta$ is an eigenvalue of $\delta_{X,\Omega}$ and its multiplicity equals $ra$, which  satisfies the inequality
 $$ra \le r \left(\frac{1}{p} \cdot \frac{2g_0}{p-1}-\frac{1}{p}\right)= \frac{1}{p} \cdot \frac{2rg_0}{p-1}-\frac{r}{p}=
  \frac{1}{p} \cdot \frac{2g}{p-1}-\frac{r}{p}.$$

 Taking into account that
 $$\frac{1}{p}\frac{2g_0}{p-1}-a>0, \ \ r>M,$$
 we get
 $$ra=r \cdot \frac{1}{p}\frac{2g_0}{p-1} -r \cdot \left(\frac{1}{p}\frac{2g_0}{p-1}-a\right)< \frac{1}{p}\frac{2g_0r}{p-1}-M
 \left(\frac{1}{p}\frac{2g_0}{p-1}-a\right)=\frac{1}{p}\frac{2g}{p-1}-\frac{p-2}{p}.$$
  In other words, $\zeta$ is an eigenvalue of  $\delta_{X,\Omega}$ with multiplicity that is strictly less than
 $$\frac{1}{p}\frac{2g}{p-1}-\frac{p-2}{p},$$
 which contradicts to Theorem \ref{bounds}. The obtained contradiction implies that $X$ is not isomorphic to a jacobian. 
\end{proof}

\begin{ex}
\label{centralD}
Let $p$ be an odd prime and $n\ge 4$ an integer such that $p$ does not divide $n$. We have
\begin{equation}
\label{npc}
n=ap+c; \quad a,c \in \bbZ_+;  1 \le c \le p-1;  \quad 0 \le c-1 \le p-2.
\end{equation}

Let $f(x) \in \bbC[x]$ be a degree $n$ polynomial without repeated roots. Let $\cV_{f,p}$ be the smooth
projective model of the smooth plane affine curve $y^p=f(x)$.  The genus $g_0$ of $\cV_{f,p}$ is $(n-1)(p-1)/2$; hence
$$\frac{2g_0}{p-1}=n-1=ap+(c-1).$$
There is an automorphism
$\tilde\delta \in \mathrm{Aut}(\cV_{f,p})$ of $\cV_{f,p}$ defined  by 
$$(x,y) \mapsto (x,  \zeta_p y).$$
Let $(J(\cV_{f,p}), \Theta)$ be the  canonically principally polarized jacobian of $C_{f,p}$.
By Albanese functoriality $\tilde\delta$ induces  the automorphism  $\delta$ of $(\cJ,\Theta)$,   
which  satisfies the $p$th cyclotomic equation. 
   The corresponding {\sl multiplicity function} is
  $$\mathbf{a}_{\cJ(\cC_{f,p}), \delta}:  (-k \bmod p) \mapsto [nk/p] \  \ (1 \le k \le p-1),$$
  see \cite[Remark 3.7]{ZarhinCamb}.
  In particular, in light of \eqref{npc},
  \begin{equation}
   \label{cm1}\begin{aligned}
  \mathbf{a}_{\cJ(\cV_{f,p}), \delta}(-1\bmod p)=[n/p] =[(ap+c)/p]=a=\\
  \frac{(n-1)-(c-1)}{p}=\frac{1}{p}\frac{2g_0}{p-1}-\frac{(c-1)}{p}
  \ge \frac{2g_0}{p-1}-\frac{p-2}{p}.\end{aligned}
  \end{equation}
  
  Let us assume that
  $$c > 1,$$
  i.e., $p$ does {\sl not} divide $n-1$. Then either $\zeta_p^{-1}$ is not an eigenvalue of
  $$\delta_{\Omega}: \Omega^1(\cJ(\cV_{f,p})) \to \Omega^1(\cJ(\cV_{f,p})) $$ 
  or it is an eigenvalue but  its  multiplicity $a$ is strictly less than 
  $\frac{1}{p}\frac{2g_0}{p-1}$. Now it follows from Corollary \ref{powersC}
  that if $\delta$ is a {\sl central element} of $\End(\cJ(\cV_{f,p}))$ then $\cJ(\cV_{f,p})^r$ is {\sl not} isomorphic as an algebraic variety to the jacobian of a smooth connected projective curve for all positive integers
  $$r>\frac{p-2}{\frac{1}{p} \frac{2g_0}{p-1}-a}=
  \frac{p-2}{\frac{1}{p} (n-1)-a}=\frac{p-2}{\frac{1}{p} (pa+c-1)-a}=
  \frac{p-2}{(c-1)/p}=\frac{p(p-2)}{c-1}.$$
  In other words, $\cJ(\cV_{f,p})^r$ is not isomorphic to a jacobian (even if one ignores polarizations) if 
  $\delta$ is central, $c>1$, and
  $$r>\frac{p(p-2)}{c-1}.$$
\end{ex}

\begin{thm}
Let $p$ be an odd prime and $n\ge 5$ an integer. Suppose that $p$ does not divide $n(n-1)$, i.e., $n=ap+c$ where $a,c \in \bbZ_{+}$,
and $2 \le c \le p-1$.
 Let $K$ be a subfield of $\bbC$ that contains $\zeta_p$.  
 Let 
 $f(x) \in K[x] \subset \bbC[x]$
 be a degree $n$ irreducible polynomial over $K$, whose Galois group $\Gal(f/K)$ over $K$ 
 is either the full symmetric group $\mathbf{S}_n$ or  the alternating group $\mathbf{A}_n$.
 enjoys one of the following two properties.
 \begin{itemize}
  \item
  $n \ge 4$ and $\Gal(f)$ coincides with the full symmetric group $\mathbf{S}_n$;
  \item
  $n \ge 5$ and $\Gal(f)$ coincides with the alternating group $\mathbf{A}_n$.
 \end{itemize}

Then $\cJ(\cV_{f,p})^r$ is {\sl not} isomorphic as an algebraic variety to the jacobian of a smooth connected projective curve for all positive integers
$$r>\frac{p(p-2)}{c-1}.$$
In particular,  $\cJ(\cV_{f,p})^r$ is {\sl not} isomorphic to a jacobian if 
$r>p(p-2)$.
\end{thm}

\begin{proof}
 Our assumptions on $n$ and $f(x)$ imply that the endomorphism ring 
 $\End(J(\cV_{f,p}))$ equals $\bbZ[\delta]\cong \bbZ[\zeta_p]$.
 Indeed, 
 in the case  $n \ge 5$, it follows from \cite[Th. 1.1]{ZarhinCamb} (see also  corrigendum in \cite[Remark 1.4]{ZarhinA5} for $n=5$);
 in the case $n=4$, it follows from  \cite[Th. 1.3]{ZarhinKumar}, because $\mathbf{S}_4$ does not have a normal subgroup of index 3.
  
 In particular, the ring $\End(J(\cV_{f,p}))$ is commutative and therefore $\delta$ lies in its center. Now the desired results follow readily from considerations of Example \ref{centralD} if we take into account that
 $c-1 \ge 2-1=1$ and therefore
 $p(p-2) \ge  p(p-2)/(c-1)$.
 
\end{proof}

\begin{thm}[about self-products of abelian varieties of CM type]
\label{CMrev}
Let $p$ be an odd prime, $Y$ a complex abelian variety of dimension $(p-1)/2$ endowed with a ring isomorphism 
$\kappa: \bbZ[\zeta_p] \cong \End(Y)$.

 If $r>(p-2) $ is an integer then $Y^r$ is not isomorphic as an algebraic variety to the jacobian of a smooth connected projective curve.
\end{thm}
\begin{proof}
Let us consider
$\delta:=\kappa(\zeta_p) \in \Aut(Y)$.
Clearly, $\delta$ satisfies the $p$th cyclotomic equation in $\End(Y)$. On the other hand, the complex vector space $\Omega^1(Y)$ has dimension $(p-1)/2$ that is strictly less that $(p-1)$. Hence, there is a primitive $p$th root of unity, say, $\zeta$ that is {\sl not} an eigenvalue of $\delta_{\Omega}: \Omega^1(Y) \to \Omega^1(Y)$. Now the desired result follows from Corollary \ref{powersC} applied to 
$g_0=(p-1)/2$ and $a=0$.
\end{proof}

\begin{ex}
 Let $p=3$ and $Y$ an elliptic curve with $\End(Y)=\bbZ[\zeta_3]$. It follows from Theorem \ref{CMrev} applied to $p=3$  that if $r \ge 2$ is an integer 
 then $Y^r$ is {\sl not} isomorphic as an algebraic variety to the jacobian of a smooth connected projective curve. (This assertion is well known for $r=2$, see \cite{Hn}.)
\end{ex}

\end{document}